\documentclass[a4paper,english]{amsart}
\usepackage[T1]{fontenc}
\usepackage[latin9]{inputenc}
\usepackage{amsmath}
\usepackage{amsthm}
\usepackage{amssymb}

\makeatletter

\pdfpageheight\paperheight
\pdfpagewidth\paperwidth

\theoremstyle{plain}
\newtheorem{thm}{\protect\theoremname}
  \theoremstyle{remark}
  \newtheorem{rem}[thm]{\protect\remarkname}
  \theoremstyle{remark}
  \newtheorem*{rem*}{\protect\remarkname}
  \theoremstyle{remark}
  \newtheorem*{not*}{Notations}
  \theoremstyle{plain}
  \newtheorem{lem}[thm]{\protect\lemmaname}
  \theoremstyle{plain}
  \newtheorem{prop}[thm]{\protect\propositionname}
  \theoremstyle{definition}
  \newtheorem{defn}[thm]{\protect\definitionname}

\makeatother

\usepackage{babel}
  \providecommand{\definitionname}{Definition}
  \providecommand{\lemmaname}{Lemma}
  \providecommand{\propositionname}{Proposition}
  \providecommand{\remarkname}{Remark}
\providecommand{\theoremname}{Theorem}

\begin{document}

\title[Blow-up phenomena for linearly perturbed Yamabe problem]{Blow-up phenomena for linearly perturbed Yamabe problem on manifolds
with umbilic boundary}

\author{Marco Ghimenti}
\address[Marco Ghimenti]{Dipartimento di Matematica
Universit\`a di Pisa
Largo Bruno Pontecorvo 5, I - 56127 Pisa, Italy}
\email{marco.ghimenti@unipi.it }

\author{Anna Maria Micheletti}
\address[Anna Maria Micheletti]{Dipartimento di Matematica
Universit\`a di Pisa
Largo Bruno Pontecorvo 5, I - 56127 Pisa, Italy}
\email{a.micheletti@dma.unipi.it }

\author{Angela Pistoia}
\address[Angela Pistoia] {Dipartimento SBAI, Universt\`{a} di Roma ``La Sapienza", via Antonio Scarpa 16, 00161 Roma, Italy}
\email{angela.pistoia@uniroma1.it}

\thanks{The first author is partially supported by P.R.A., University of Pisa}
\begin{abstract}
We build blowing-up solutions for linear perturbation of the Ya\-ma\-be problem on manifolds with umbilic boundary, provided the Weyl tensor is nonzero everywhere on the boundary and the dimension of the manifold is $n\ge11$.

 \end{abstract}


\maketitle

\section{Introduction}

The well known Yamabe problem  consists of finding a constant scalar curvature metric  which is
pointwise conformal to a given metric $g$ on an $n$-dimensional ($n\ge3$)   compact Riemannian
manifold $M$ without boundary. From a PDE\rq{}s point of view,
this is equivalent to finding a positive solution to the semilinear
elliptic equation
\begin{equation}\label{y}
 L_{g}u=\kappa u^{n+2\over n-2}\ \text{ in }M 
\end{equation}
where $\kappa$ is a constant, $L_gu=-\Delta_gu +c(n)R_g u$ is the conformal Laplacian for $g$ with scalar
curvature $R_g$ and $c(n):={n-2\over 4(n-1)}$. Indeed,  if $u$ is a positive solution of  \eqref{y}, then the new metric
$\tilde g=u^{4\over n-2} g$
has scalar curvature $c(n) \kappa.$

This problem has been complete solved  through the combined
works of Yamabe \cite{yam}, Trudinger \cite{tru}, Aubin \cite{aub} and Schoen \cite{S}. The structure of the full set of solutions of \eqref{y} has also been completely understood. We quote the survey of  Brendle and Marques \cite{BM} for 
a complete overview on the   compactness and non-compactness results. 
A related issue is the compactness of {\em linear} or {\em non-linear perturbations} of problem \eqref{y}  which has been largely studied in the last few years with contributions by several authors (see \cite{D,DH,EP,EPV,MPV,MPVa,PV,RV}).
\bigskip

An obvious extension of such problems is to consider manifolds with boundary. The
Yamabe problem on manifolds with boundary was initially investigated  by Escobar  \cite{E1,E2}.
In this
case one would like to find  a  metric $g$ on an $n$-dimensional ($n\ge3$)   compact Riemannian
manifold $M$ with  boundary $\partial M$  which has not only constant
scalar curvature but constant mean curvature as well. This problem is equivalent to showing
the existence of a positive solution to the boundary value problem
\begin{equation}
\left\{ \begin{array}{cc}
L_{g}u=\kappa u^{n+2\over n-2} & \text{ in }M\\
\partial_{\nu_g}u+\frac{n-2}{2}h_{g}u=\mathfrak c u^{\frac{n}{n-2}} & \text{ on }\partial M
\end{array}\right.\label{esc}
\end{equation}
where ${\nu_g}$ is the unit outer normal and $h_{g}$  is the mean curvature. If such a solution exists, then the
metric $\tilde g=u^{4\over n-2} g$
has scalar curvature $c(n) \kappa$
 and the boundary has mean curvature $\mathfrak c.$
 Problem \eqref{esc} has been solved starting from 
  Escobar in \cite{E1,E2} with contributions from several authors when either $\kappa\not=0$ and $\mathfrak c=0$ or $\kappa=0$ and $\mathfrak c\not=0$ (see the recent paper by Disconzi and Khuri \cite{DK} for an exhaustive list of references)

  In this paper we will focus on the zero scalar curvature case, i.e. $\kappa=0,$ so problem \eqref{esc} reduces to finding a positive solution
 to the boundary value problem 
\begin{equation}
\left\{ \begin{array}{cc}
L_{g}u=0 & \text{ in }M\\
\partial_{\nu}u+\frac{n-2}{2}h_{g}u=\mathfrak c u^{\frac{n}{n-2}} & \text{ on }\partial M.
\end{array}\right.\label{eq:probK}
\end{equation}
 
Solutions to (\ref{eq:probK}) are critical points of the functional
\[
Q(u):=\frac{\int\limits _{M}\left(|\nabla u|^{2}+\frac{n-2}{4(n-1)}R_{g}u^{2}\right)dv_{g}+\int\limits _{\partial M}\frac{n-2}{2}h_{g}u^{2}d\sigma_{g}}{\left(\int\limits _{\partial M}|u|^{\frac{2(n-1)}{n-2}}d\sigma\right)^{\frac{n-2}{n-1}}},\ u\in H
\]
where $dv_{g}$ and $d\sigma_{g}$ denote the volume forms on $M$
and $\partial M,$ respectively, and the space 
\[
H:=\left\{ u\in H_{g}^{1}(M)\ :\ u\not=0\text{ on }\partial M\right\} .
\]
Escobar in \cite{E1} introduced the Sobolev quotient 
\begin{equation}
Q(M,\partial M):=\inf\limits _{H}Q(u),\label{sob-quo}
\end{equation}
which is conformally invariant and always satisfies 
\begin{equation}
Q(M,\partial M)\le Q(\mathbb{B}^{n},\partial\mathbb{B}^{n}),\label{ineq}
\end{equation}
where $\mathbb{B}^{n}$ is the unit ball  in $\mathbb{R}^{n}$ endowed
with the euclidean metric $\mathfrak{g}_{0}$.
Following Aubin's approach (see \cite{aub}), Escobar proved that
if $Q(M,\partial M)$ is finite and the strict inequality in (\ref{ineq})
holds, i.e. 
\begin{equation}
Q(M,\partial M)<Q(\mathbb{B}^{n},\partial\mathbb{B}^{n}),\label{strict}
\end{equation}
then the infimum (\ref{sob-quo}) is achieved and a solution to problem
(\ref{eq:probK}) does exist.
In the negative case, i.e. $Q(M,\partial M)\le0$, it is quite easy to prove that
(\ref{strict}) holds. The positive case, i.e. $Q(M,\partial M)>0$,
is the most difficult one and the proof of the validity of (\ref{strict})
 required a lot of work. When $(M,g)$ is not conformally equivalent
to $(\mathbb{B}^{n},\mathfrak{g}_{0})$, (\ref{strict}) has been
proved by Escobar in \cite{E1}, by Marques in \cite{M1,Ma2} and by
Almaraz in \cite{A1}. 

Once the existence of solutions of problem (\ref{eq:probK}) is settled,
a natural question concerns the structure of the full set of positive
solutions of (\ref{eq:probK}).
If $Q(M,\partial M)<0$ the solution is unique and if $Q(M,\partial M)=0$
the solution is unique up to a constant factor. If $Q(M,\partial M)>0$
the situation turns out to be more delicate. Indeed, the round hemisphere 
provides the canonical example of non compactness, while compactness was proved by Felli and Ould-Ahmedou in \cite{FO03}  when   $(M,g)$ is locally conformally
flat and $\partial M$ is umbilic and by  Almaraz in \cite{Al}, when $n\ge7$ and the trace-free second fundamental
form of $\partial M$ is non zero everywhere.
Up to our knowledge, the only non-compactness result is due to Almaraz 
in \cite{A2}, where he constructs a sequence of blowing-up conformal metrics
with zero scalar curvature and constant boundary mean curvature on
a ball of dimension $n\ge25.$ It is unknown if the dimension $25$
is sharp for the compactness, namely if $n\le24$ the problem (\ref{eq:probK})
is compact or not.

The compactness issue is closely related to the existence of blowing-up solutions for {\em small perturbations} of problem \eqref{eq:probK}. In particular, we consider the {\em linear perturbation  problem}
\begin{equation}
\left\{ \begin{array}{cc}
-\Delta_{g}v+\frac{n-2}{4(n-1)}R_{g}v=0 & \text{ in }M\\
\frac{\partial v}{\partial\nu}+\frac{n-2}{2}h_{g}v+\varepsilon\gamma v=(n-2)v^{\frac{n}{n-2}} & \text{ on }\partial M
\end{array}\right.\label{eq:P}
\end{equation}
 where $\varepsilon$ is a small positive parameter and $\gamma$ is a given smooth function,
 and we address the following question:
 \begin{itemize}
 \item[(Q)] {\em Does problem \eqref{eq:P} have a family of solutions which blows up at one point of the manifold as $\varepsilon$ approaches zero?}
 \end{itemize}
 A first positive answer was given by the authors in  \cite{GMP} when $n\ge7$ and the boundary is not umbilic. In the present paper, we give a positive answer  when the boundary is umbilic. 
   Our main result reads as follows.
\begin{thm}
\label{thm:main}Let $(M,g)$ be a smooth, $n$-dimensional Riemannian
manifold of positive type with regular umbilic boundary $\partial M$.
Suppose that $n\ge11$ and that the Weyl tensor is not vanishing on
$\partial M$. Let $\gamma:M\rightarrow\mathbb{R}$ a smooth function, $\gamma>0$
on $\partial M$. Then, for $\varepsilon>0$ small there exists a positive
solution $v_{\varepsilon}$ of the problem \eqref{eq:P}
 such that $v_{\varepsilon}$ blows up at a suitable point $q_{0}\in\partial M$
as $\varepsilon\rightarrow0$.
\end{thm}

Let us make some comments on our result.
\begin{itemize}
\item[(1)]
The proof relies on the classical finite dimensional Ljapunov-Schmidt procedure which has been successfully used in 
studying blowing-up phenomena in Yamabe type problems.
However, here the umbilicity of the boundary forces us to deal with higher order terms in the expansion of the metric $g$, which makes the proof of the result technically harder than the one in \cite{GMP}.

\item[(2)] Our theorem  does not provide the precise location of the blow-up point, because the explicit solution 
to linear problem \eqref{eq:vqdef} is necessary and this is far from being possible.
Actually, it would be really interesting to detect the geometric function whose critical points 
generate the blowing-up solutions.

\item[(3)] We believe that the result holds true if $\gamma$ is positive
somewhere (and not necessarily positive everywhere in $\partial M$) as suggested by  Remark
\ref{rem:esempio}, where we 
  exhibit a smooth function $\gamma$ which is not necessarily everywhere positive, for which problem
(\ref{eq:P}) has a family of blowing-up solutions.
 Actually, we strongly believe that   if $\gamma$ is negative everywhere there are no blowing-up solutions as $\varepsilon$ approaches 0, i.e. the problem (\ref{eq:P})
is compact.

\item[(4)]
Our ideas can be also applied to study the {\em non-linear perturbation problem}
\begin{equation}
\left\{ \begin{array}{cc}
-\Delta_{g}v+\frac{n-2}{4(n-1)}R_{g}v=0 & \text{ in }M\\
\frac{\partial v}{\partial\nu}+\frac{n-2}{2}h_{g}v =(n-2)v^{\frac{n}{n-2}+ \varepsilon} & \text{ on }\partial M
\end{array}\right.\label{eq:nP}
\end{equation}
In particular, we can extend the results of the authors  in \cite{GMP16} to the geometric
problem \eqref{eq:nP}. 

\end{itemize}
 
 The proof of the result relies on a finite dimensional Ljapunov-Schmidt reduction, which is
 carried out as usual through  different steps: first we find a good approximated solution (Section \ref{preli}), next we reduce the problem to a finite dimensional one (Section \ref{sec:red}), then we study the reduced problem (Section \ref{reduced})
 and finally we complete the proof of Theorem \ref{thm:main} (Section \ref{completed}).

\section{Preliminaries and variational framework}\label{preli}
\begin{not*}
We collect here our main notations. We will use the indices $1\le i,j,k,m,p,r,s\le n-1$
and $1\le a,b,c,d\le n$. We denote by $g$ the Riemannian metric,
by $R_{abcd}$ the full Riemannian curvature tensor, by $R_{ab}$
the Ricci tensor and by $R_{g}$ the scalar curvature of $(M,g)$;
moreover the Weyl tensor of $(M,g)$ will be denoted by $W_{g}$.

Let $\left(h_{ij}\right)_{ij}(q)$
be the tensor of the second fundamental form in a point $q\in\partial M$.
We recall that the boundary $\partial M$ is umbilic (i.e. composed only of umbilic points) when, for all $q\in \partial M$, $h_{ij}(q)=0$ 
for all $i\neq j$ and $h_{ii}(q)=h_{g}(q)$,
$h_{g}(q)$ being the mean curvature of $\partial M$ at the point
$q$. 

The bar over an object (e.g. $\bar{W}_{g}$) will mean the restriction
to this object to the metric of $\partial M$. By $-\Delta_{g}$ we
denote the Laplace-Beltrami operator on $(M,g)$ and we will often
use the common notation for conformal Laplacian $L_{g}=-\Delta_{g}+\frac{n-2}{4(n-1)}R_{g}$
and the conformal boundary operator $B_{g}=\frac{\partial}{\partial\nu}+\frac{n-2}{2}h_{g}$,
where $\nu$ is the outward normal to $\partial M$ . When we derive a tensor, e.g. $T_{ij}$, with respect to a coordinate $y_l$ we use the usual shortened notation 
$T_{ij,l}$ for $\frac{\partial}{\partial y_l}T_{ij}$.
\end{not*}
\begin{rem}
Since $\partial M$ is umbilic for any $q\in\partial M$, there exists
a metric $\tilde{g}_{q}=\tilde{g}$, conformal to $g$, $\tilde{g}_{q}=\Lambda_{q}^{\frac{4}{n-2}}g_{q}$
such that
\begin{equation}
|\text{det}\tilde{g}_{q}(y)|=1+O(|y|^{n})\label{eq:|g|}
\end{equation}
\begin{equation}
|\tilde{h}_{ij}(y)|=o(|y^{3}|)\label{eq:hij}
\end{equation}
\begin{align}
\tilde{g}^{ij}(y)= & \delta_{ij}+\frac{1}{3}\bar{R}_{ikjl}y_{k}y_{l}+R_{ninj}y_{n}^{2}\label{eq:gij}\\
 & +\frac{1}{6}\bar{R}_{ikjl,m}y_{k}y_{l}y_{m}+R_{ninj,k}y_{n}^{2}y_{k}+\frac{1}{3}R_{ninj,n}y_{n}^{3}\nonumber \\
 & +\left(\frac{1}{20}\bar{R}_{ikjl,mp}+\frac{1}{15}\bar{R}_{iksl}\bar{R}_{jmsp}\right)y_{k}y_{l}y_{m}y_{p}\nonumber \\
 & +\left(\frac{1}{2}R_{ninj,kl}+\frac{1}{3}\text{Sym}_{ij}(\bar{R}_{iksl}R_{nsnj})\right)y_{n}^{2}y_{k}y_{l}\nonumber \\
 & +\frac{1}{3}R_{ninj,nk}y_{n}^{3}y_{k}+\frac{1}{12}\left(R_{ninj,nn}+8R_{nins}R_{nsnj}\right)y_{n}^{4}+O(|y|^{5})\nonumber 
\end{align}
\begin{equation}
\bar{R}_{\tilde{g}_{q}}(y)=O(|y|^{2})\text{ and }\partial_{ii}^{2}\bar{R}_{\tilde{g}_{q}}(q)=-\frac{1}{6}|\bar{W}(q)|^{2}\label{eq:Rii}
\end{equation}
\begin{equation}
\bar{R}_{kl}(q)=R_{nn}(q)=R_{nk}(q)=0\label{eq:Ricci}
\end{equation}
uniformely with respect to $q\in M$ and $y\in T_{q}(M)$. Also,we
have $\Lambda_{q}(q)=1$ and $\nabla\Lambda_{q}(q)=0$. This results
are contained in \cite{M1,KMW}.
\end{rem}
The conformal Laplacian and the conformal boundary operator transform
under the change of metric $\tilde{g}_{q}=\Lambda_{q}^{\frac{4}{n-2}}g_{q}$ 
in the following way:
\begin{align*}
L_{\tilde{g}_{q}}\varphi & =\Lambda_{q}^{-\frac{n+2}{n-2}}L_{g}(\Lambda_{q}\varphi)\\
B_{\tilde{g}_{q}}\varphi & =\Lambda_{q}^{-\frac{n}{n-2}}B_{g}(\Lambda_{q}\varphi)
\end{align*}
by these transformations we can recast Problem (\ref{eq:P}) as
follows: $v:=\Lambda_{q}u$ is a positive solution of (\ref{eq:P}),
if and only if $u$ is a positive solution of 
\begin{equation}
\left\{ \begin{array}{cc}
L_{\tilde{g}_{q}}u=0 & \text{ in }M\\
B_{\tilde{g}_{q}}u+\varepsilon[\Lambda_{q}^{-\frac{2}{n-2}}\gamma]u=(n-2)u^{\frac{n}{n-2}} & \text{ on }\partial M
\end{array}\right.\label{eq:Ptilde}
\end{equation}
From now on we set $\tilde{\gamma}=\Lambda_{q}^{-\frac{2}{n-2}}\gamma$.

We want to find a solution $u$ of problem (\ref{eq:Ptilde}) by a
finite dimensional reduction: we will look for a solution of (\ref{eq:Ptilde})
of the form $u=W_{\delta,q}+\delta^{2}V_{\delta,q}+\Phi$ where $W_{\delta,q}$
and $V_{\delta,q}$ are functions depending only by $q\in\partial M$
and $\delta>0$ which will be defined in the following and $\Phi$
is a suitable remainder term. So we will find a solution of the original
problem (\ref{eq:P}) of the type 
\[
v=\Lambda_{q}\left[W_{\delta,q}+\delta^{2}V_{\delta,q}+\Phi\right].
\]
In the following we simply use $\tilde{W}_{\delta,q}$, $\tilde{V}_{\delta,q}$,
$\tilde{\Phi}$ respectively for $\Lambda_{q}W_{\delta,q}$, $\Lambda_{q}V_{\delta,q}$,
$\Lambda_{q}\Phi$.

If $Q(M,\partial M)>0$, we can endow $H_{g}^{1}(M)=H^1(M)$ with the following
equivalent scalar product 
\begin{equation}
\left\langle \left\langle u,v\right\rangle \right\rangle _{g}=\int_{M}(\nabla_{g}u\nabla_{g}v+\frac{n-2}{4(n-1)}R_{g}uv)d\mu_{g}+\frac{n-2}{2}\int_{\partial M}h_{g}uvd\nu_{g}\label{eq:prodscal}
\end{equation}
which leads to the norm $\|\cdot\|_{g}$ equivalent to the usual one.
We remark also that $\Lambda_{q}$ is an isometry in the sense that,
by (\ref{eq:prodscal}), for $u,v\in H^{1}(M)$
\[
\left\langle \left\langle \Lambda_{q}u,\Lambda_{q}v\right\rangle \right\rangle _{g}=\left\langle \left\langle u,v\right\rangle \right\rangle _{\tilde{g}_{q}}\text{ and, consequently, }\|\Lambda_{q}u\|_{g}=\|u\|_{\tilde{g}_{q}}.
\]

Given $q\in\partial M$ and $\psi_{q}^{\partial}:\mathbb{R}_{+}^{n}\rightarrow M$
the Fermi coordinates of $M$ defined in a neighborhood of $q$ we
set
\begin{align*}
W_{\delta,q}(\xi):= & U_{\delta}\left(\left(\psi_{q}^{\partial}\right)^{-1}(\xi)\right)\chi\left(\left(\psi_{q}^{\partial}\right)^{-1}(\xi)\right)\\
= & \frac{1}{\delta^{\frac{n-2}{2}}}U\left(\frac{y}{\delta}\right)\chi(y)=\frac{1}{\delta^{\frac{n-2}{2}}}U\left(x\right)\chi(\delta x)
\end{align*}
where $y=(z,t)$, with $z\in\mathbb{R}^{n-1}$ and $t\ge0$, $\delta x=y=\left(\psi_{q}^{\partial}\right)^{-1}(\xi)$
and $\chi$ is a radial cut off function, with support in ball of
radius $R$, $R$ being the injectivity radius for the Fermi coordinates. 

Here $U_{\delta}(y)=\frac{1}{\delta^{\frac{n-2}{2}}}U\left(\frac{y}{\delta}\right)$
is the one parameter family of solution of the problem
\begin{equation}
\left\{ \begin{array}{ccc}
-\Delta U_{\delta}=0 &  & \text{on }\mathbb{R}_{+}^{n};\\
\frac{\partial U_{\delta}}{\partial t}=-(n-2)U_{\delta}^{\frac{n}{n-2}} &  & \text{on \ensuremath{\partial}}\mathbb{R}_{+}^{n},
\end{array}\right.\label{eq:Udelta}
\end{equation}
 that is ${\displaystyle U(z,t):=\frac{1}{\left[(1+t)^{2}+|z|^{2}\right]^{\frac{n-2}{2}}}}$
is the standard bubble in $\mathbb{R}_{+}^{n}$.

Now, if we consider the linearized problem 
\begin{equation}
\left\{ \begin{array}{ccc}
 & -\Delta\phi=0 & \text{on }\mathbb{R}_{+}^{n},\\
 & \frac{\partial\phi}{\partial t}+nU^{\frac{2}{n-2}}\phi=0 & \text{on \ensuremath{\partial}}\mathbb{R}_{+}^{n},\\
 & \phi\in H^{1}(\mathbb{R}_{+}^{n}).
\end{array}\right.\label{eq:linearizzato}
\end{equation}
we have that every solution of (\ref{eq:linearizzato}) is a linear
combination of the functions $j_{1},\dots,j_{n}$ defined by 
\begin{eqnarray}
j_{i}=\frac{\partial U}{\partial y_{i}},\ i=1,\dots n-1; &  & j_{n}=\frac{n-2}{2}U+\sum_{a=1}^{n}y_{a}\frac{\partial U}{\partial y_{a}}\label{eq:sol-linearizzato}
\end{eqnarray}
(for a proof of this result, see, for instance, \cite[Lemma 6]{GMP}).
By means of functions $j_{i}$ we define we define, for $b=1,\dots,n$
\[
Z_{\delta,q}^{b}(\xi)=\frac{1}{\delta^{\frac{n-2}{2}}}j_{b}\left(\frac{1}{\delta}\left(\psi_{q}^{\partial}\right)^{-1}(\xi)\right)\chi\left(\left(\psi_{q}^{\partial}\right)^{-1}(\xi)\right)
\]
and we decompose $H^{1}(M)$ in the direct sum of the following two
subspaces 
\begin{align*}
\tilde{K}_{\delta,q} & =\text{Span}\left\langle \Lambda_{q}Z_{\delta,q}^{1},\dots,\Lambda_{q}Z_{\delta,q}^{n}\right\rangle \\
\tilde{K}_{\delta,q}^{\bot} & =\left\{ \varphi\in H^{1}(M)\ :\ \left\langle \left\langle \varphi,\Lambda_{q}Z_{\delta,q}^{b}\right\rangle \right\rangle _{g}=0,\ b=1,\dots,n\right\} 
\end{align*}
and we define the projections 
\[
\tilde{\Pi}=H^{1}(M)\rightarrow\tilde{K}_{\delta,q}\text{ and }\tilde{\Pi}^{\bot}=H^{1}(M)\rightarrow\tilde{K}_{\delta,q}^{\bot}.
\]

Given $q\in\partial M$ we also define in a similar way 
\[
V_{\delta,q}(\xi)=\frac{1}{\delta^{\frac{n-2}{2}}}v_{q}\left(\frac{1}{\delta}\left(\psi_{q}^{\partial}\right)^{-1}(\xi)\right)\chi\left(\left(\psi_{q}^{\partial}\right)^{-1}(\xi)\right)
\]
and 
\[
\left(v_{q}\right)_{\delta}(y)=\frac{1}{\delta^{\frac{n-2}{2}}}v_{q}\left(\frac{y}{\delta}\right);
\]
here $v_{q}:\mathbb{R}_{+}^{n}\rightarrow\mathbb{R}$ is the solution
of the linear problem 
\begin{equation}
\left\{ \begin{array}{ccc}
-\Delta v=\left[\frac{1}{3}\bar{R}_{ijkl}(q)y_{k}y_{l}+R_{ninj}(q)y_{n}^{2}\right]\partial_{ij}^{2}U &  & \text{on }\mathbb{R}_{+}^{n}\\
\frac{\partial v}{\partial y_{n}}=-nU^{\frac{2}{n-2}}v &  & \text{on }\partial\mathbb{R}_{+}^{n}
\end{array}\right.\label{eq:vqdef}
\end{equation}
These solutions will be used in the blow up estimate in the next:
indeed, by means of the choice of $v_{q}$ we will be able to cancel
the first order term in the following formula (\ref{eq:deltaw+v})
and to have the correct size of the remainder term in the finite dimensional
reduction (Lemma \ref{lem:R}).
\begin{lem}
\label{lem:vq}There exists a unique $v_{q}:\mathbb{R}_{+}^{n}\rightarrow\mathbb{R}$
solution of the problem (\ref{eq:vqdef}) $L^{1}(\mathbb{R}_{+}^{n})$-ortogonal
to $j_{b}$ for all $b=1,\dots,n$. Moreover the function $q\mapsto v_{q}$
is $C^{2}(\partial M)$ and it holds
\begin{equation}
|\nabla^{\tau}v_{q}(y)|\le C(1+|y|)^{4-\tau-n}\text{ for }\tau=0,1,2,\label{eq:gradvq}
\end{equation}
\begin{equation}
\int_{\partial\mathbb{R}_{+}^{n}}U^{\frac{n}{n-2}}(t,z)v_{q}(t,z)dz=0\label{eq:Uvq}
\end{equation}
and 
\begin{equation}
\int_{\partial\mathbb{R}_{+}^{n}}v_{q}(t,z)\Delta v_{q}(t,z)dz\le0,\label{new}
\end{equation}
where $y\in\mathbb{R}_{+}^{n}$, $y=(t,z)$ with $t\ge0$ and $z\in\mathbb{R}^{n-1}$.
\end{lem}
The proof of this result is postponed to Appendix.

We have the well know maps: 
\begin{align*}
i_{g}: & H^{1}(M)\rightarrow L^{t}(\partial M)\\
i_{g}^{*}: & L^{t'}(\partial M)\rightarrow H^{1}(M)
\end{align*}
for $1\le t\le\frac{2(n-1)}{n-2}$ (and for $1\le t<\frac{2(n-1)}{n-2}$
the embedding $i$ is compact). 

Given $f\in L^{\frac{2(n-1)}{n-2}}(\partial M)$ there exists a unique
$v\in H^{1}(M)$ such that 
\begin{align}
v=i_{g}^{*}(f) & \iff\left\langle \left\langle v,\varphi\right\rangle \right\rangle _{g}=\int_{\partial M}f\varphi d\sigma\text{ for all }\varphi\label{eq:istella}\\
 & \iff\left\{ \begin{array}{ccc}
-\Delta_{g}v+\frac{n-2}{4(n-1)}R_{g}v=0 &  & \text{on }M;\\
\frac{\partial v}{\partial\nu}+\frac{n-2}{2}h_{g}v=f &  & \text{on \ensuremath{\partial}}M.
\end{array}\right.\nonumber 
\end{align}
 The functional defined on $H^{1}(M)$ associated to (\ref{eq:P})
is 
\begin{align*}
J_{\varepsilon,g}(v): & =\frac{1}{2}\int_{M}|\nabla_{g}v|^{2}+\frac{n-2}{4(n-1)}R_{g}v^{2}d\mu_{g}+\frac{n-2}{4}\int_{\partial M}h_{g}v^{2}d\sigma_{g}\\
 & +\frac{1}{2}\int_{\partial M}\varepsilon\gamma v^{2}d\sigma_{g}-\frac{(n-2)^{2}}{2(n-1)}\int_{\partial M}\left(v^{+}\right)^{\frac{2(n-1)}{n-2}}d\sigma_{g}.
\end{align*}
Notice that, if we define 
\begin{align*}
\tilde{J}_{\varepsilon,\tilde{g}_{q}}(u): & =\frac{1}{2}\int_{M}|\nabla_{\tilde{g}_{q}}u|^{2}+\frac{n-2}{4(n-1)}R_{\tilde{g}_{q}}u^{2}d\mu_{\tilde{g}_{q}}+\frac{n-2}{4}\int_{\partial M}h_{\tilde{g}_{q}}u^{2}d\sigma_{\tilde{g}_{q}}\\
 & +\frac{1}{2}\int_{\partial M}\varepsilon\tilde{\gamma}u^{2}d\sigma_{\tilde{g}_{q}}-\frac{(n-2)^{2}}{2(n-1)}\int_{\partial M}\left(u^{+}\right)^{\frac{2(n-1)}{n-2}}d\sigma_{v},
\end{align*}
then we have
\begin{equation}
J_{\varepsilon,g}(\Lambda_{q}u)=\tilde{J}_{\varepsilon,\tilde{g}_{q}}(u)\label{eq:Jlambda}
\end{equation}

\section{\label{sec:red}The finite dimensional reduction}

Solving problem (\ref{eq:P}) is equivalent to find $v\in H^{1}(M)$
such that 

\[
v=i_{g}^{*}(f(v)-\varepsilon\gamma v)
\]
where 
\[
f(v)=(n-2)\left(v^{+}\right)^{\frac{n}{n-2}}
\]
We remark that, if $v\in H_{g}^{1}(M)$, then $f(v)\in L^{\frac{2(n-1)}{n}}(\partial M)$. 

We look for a positive solution of (\ref{eq:P}) in the form 
\[
v=\Lambda_{q}u=\tilde{W}_{\delta,q}+\delta^{2}\tilde{V}_{\delta,q}+\tilde{\Phi}
\]
(we recall that, if $f:M\rightarrow\mathbb{R}$, we use the notation
$\tilde{f}:=\Lambda_{q}f$). Thus we can rewrite, in light of the
previous orthogonal decomposition, Problem (\ref{eq:P}) as 
\begin{align}
\tilde{\Pi}\left\{ \tilde{W}_{\delta,q}+\delta^{2}\tilde{V}_{\delta,q}+\tilde{\Phi}-i_{g}^{*}\left[f(\tilde{W}_{\delta,q}+\delta^{2}\tilde{V}_{\delta,q}+\tilde{\Phi})-\varepsilon\gamma\tilde{W}_{\delta,q}+\delta^{2}\tilde{V}_{\delta,q}+\tilde{\Phi}\right]\right\}  & =0\label{eq:Pi}\\
\tilde{\Pi}^{\bot}\left\{ \tilde{W}_{\delta,q}+\delta^{2}\tilde{V}_{\delta,q}+\tilde{\Phi}-i_{g}^{*}\left[f(\tilde{W}_{\delta,q}+\delta^{2}\tilde{V}_{\delta,q}+\tilde{\Phi})-\varepsilon\gamma\tilde{W}_{\delta,q}+\delta^{2}\tilde{V}_{\delta,q}+\tilde{\Phi}\right]\right\}  & =0.\label{eq:Pibot}
\end{align}
Now we define the linear operator $L:\tilde{K}_{\delta,q}^{\bot}\rightarrow\tilde{K}_{\delta,q}^{\bot}$
as
\begin{equation}
L(\tilde{\Phi})=\tilde{\Pi}^{\bot}\left\{ \tilde{\Phi}-i_{g}^{*}\left(f'(\tilde{W}_{\delta,q}+\delta^{2}\tilde{V}_{\delta,q})[\tilde{\Phi}]\right)\right\} ,\label{eq:defL}
\end{equation}
we define a nonlinear term $N(\tilde{\Phi})$ and a remainder term
R as 
\begin{align}
N(\tilde{\Phi})= & \tilde{\Pi}^{\bot}\left\{ i_{g}^{*}\left(f(\tilde{W}_{\delta,q}+\delta^{2}\tilde{V}_{\delta,q}+\tilde{\Phi})-f(\tilde{W}_{\delta,q}+\delta^{2}\tilde{V}_{\delta,q})-f'(\tilde{W}_{\delta,q}+\delta^{2}\tilde{V}_{\delta,q})[\tilde{\Phi}]\right)\right\} \label{eq:defN}\\
R= & \tilde{\Pi}^{\bot}\left\{ i_{g}^{*}\left(f(\tilde{W}_{\delta,q}+\delta^{2}\tilde{V}_{\delta,q})\right)-\tilde{W}_{\delta,q}-\delta^{2}\tilde{V}_{\delta,q}\right\} ,\label{eq:defR}
\end{align}
so equation (\ref{eq:Pibot}) becomes

\[
L(\tilde{\Phi})=N(\tilde{\Phi})+R-\tilde{\Pi}^{\bot}\left\{ i_{g}^{*}\left(\varepsilon\gamma(\tilde{W}_{\delta,q}+\delta^{2}\tilde{V}_{\delta,q}+\tilde{\Phi})\right)\right\} .
\]

\begin{lem}
\label{lem:R}Assume $n\ge10$, then it holds 
\[
\|R\|_{g}=O\left(\varepsilon\delta+\delta^{3}\right)
\]
$C^{0}$-uniformly for $q\in\partial M$.
\end{lem}
\begin{proof}
We recall that there is a unique $\Gamma$ such that 
\[
\Gamma=i_{g}^{*}\left(f(\tilde{W}_{\delta,q}+\delta^{2}\tilde{V}_{\delta,q})\right),
\]
that is, according to (\ref{eq:istella}) equivalent to say that there
exists a unique $\Gamma$ solving 
\[
\left\{ \begin{array}{ccc}
-\Delta_{g}\Gamma+\frac{n-2}{4(n-1)}R_{g}\Gamma=0 &  & \text{on }M;\\
\frac{\partial\Gamma}{\partial\nu}+\frac{n-2}{2}h_{g}\Gamma=(n-2)\left((\tilde{W}_{\delta,q}+\delta^{2}\tilde{V}_{\delta,q})^{+}\right)^{\frac{n}{n-2}} &  & \text{on \ensuremath{\partial}}M.
\end{array}\right.
\]
By definition of $i_{g}^{*}$ we have that 
\begin{align*}
\|R\|_{g}^{2}= & \|\Gamma-\tilde{W}_{\delta,q}-\delta^{2}\tilde{V}_{\delta,q}\|_{g}^{2}\\
= & \int_{M}-\left[L_{g}(\Gamma-\tilde{W}_{\delta,q}-\delta^{2}\tilde{V}_{\delta,q})\right](\Gamma-\tilde{W}_{\delta,q}-\delta^{2}\tilde{V}_{\delta,q})d\mu_{g}\\
 & +\int_{\partial M}\left[B_{g}(\Gamma-\tilde{W}_{\delta,q}-\delta^{2}\tilde{V}_{\delta,q})\right](\Gamma-\tilde{W}_{\delta,q}-\delta^{2}\tilde{V}_{\delta,q})d\sigma_{g}\\
= & \int_{M}\left[-L_{g}(\tilde{W}_{\delta,q}+\delta^{2}\tilde{V}_{\delta,q})\right]Rd\mu_{g}\\
 & +\int_{\partial M}\left[B_{g}(\Gamma)-B_{g}(\tilde{W}_{\delta,q}+\delta^{2}\tilde{V}_{\delta,q})\right]Rd\sigma_{g}\\
= & \int_{M}\left[-L_{\tilde{g}_{q}}(W_{\delta,q}+\delta^{2}V_{\delta,q})\right]\Lambda_{q}^{-1}Rd\mu_{\tilde{g}_{q}}\\
 & +\int_{\partial M}\left[B_{\tilde{g}_{q}}(\Lambda_{q}^{-1}\Gamma)-B_{\tilde{g}_{q}}(W_{\delta,q}+\delta^{2}V_{\delta,q})\right]\Lambda_{q}^{-1}Rd\sigma_{\tilde{g}_{q}}
\end{align*}
We estimate 
\begin{align*}
\int_{M}\left[\Delta_{\tilde{g}_{q}}(W_{\delta,q}+\delta^{2}V_{\delta,q})\right]\Lambda_{q}^{-1}Rd\mu_{\tilde{g}_{q}} & \le\|\Delta_{\tilde{g}_{q}}(W_{\delta,q}+\delta^{2}V_{\delta,q})\|_{L^{\frac{2n}{n+2}}(M,\tilde{g}_{q})}\|\Lambda_{q}^{-1}R\|_{\tilde{g}_{q}}\\
 & \le O(\delta^{3})\|R\|_{g}
\end{align*}
In fact, recalling the expression fot Laplace Beltrami operator in
local charts 
\begin{align*}
\Delta_{\tilde{g}_{q}} & =\Delta_{\text{euc}}+[\tilde{g}_{q}^{ij}(y)-\delta_{ij}]\partial_{ij}^{2}\\
 & +\left[\partial_{i}\tilde{g}_{q}^{ij}(y)+\frac{\tilde{g}_{q}^{ij}(y)\partial_{i}|\tilde{g}_{q}|^{\frac{1}{2}}(y)}{|\tilde{g}_{q}|^{\frac{1}{2}}(y)}\right]\partial_{j}+\frac{\partial_{n}|\tilde{g}_{q}|^{\frac{1}{2}}(y)}{|\tilde{g}_{q}|^{\frac{1}{2}}(y)}\partial_{n}
\end{align*}
where $i,k=1,\dots,n-1$ and $\Delta_{\text{euc}}$ is the euclidean
Laplacian, by (\ref{eq:|g|}) and (\ref{eq:gij}) and since $\Delta_{\text{euc}}U=0$,
in variables $y=\delta x$ we have 
\begin{align}
\Delta_{\tilde{g}_{q}}W_{\delta,q} & =\frac{1}{\delta^{\frac{n-2}{2}}}\frac{1}{3}\left(\bar{R}_{ikjl}x_{k}x_{l}+R_{ninj}x_{n}^{2}+\delta O(|x|^{3})\right)\partial_{ij}^{2}\left(U(x)\chi(\delta x)\right)\nonumber \\
 & +\frac{1}{\delta^{\frac{n-2}{2}}}\frac{1}{3}\left(\bar{R}_{iijl}x_{l}+\bar{R}_{ikji}x_{k}+\delta O(|x|^{2})\right)\partial_{j}\left(U(x)\chi(\delta x)\right)\nonumber \\
 & +\frac{1}{\delta^{\frac{n-2}{2}}}\delta^{2}O(|x|^{3})\partial_{n}\left(U(x)\chi(\delta x)\right)\label{eq:DW}
\end{align}
We remark that, by symmetry, $\bar{R}_{iijl}=0$ and by \cite[Prop 3.2 (2)]{M1}
also $\bar{R}_{ikji}=-\bar{R}_{jk}=0$. Now, the definition of $v_{q}$
(\ref{eq:vqdef}) is crucial to get
\begin{align*}
\delta^{2}\Delta_{\tilde{g}_{q}}V_{\delta,q} & =\frac{1}{\delta^{\frac{n-2}{2}}}\Delta_{\text{euc}}\left(v_{q}(x)\chi(\delta x)\right)+\frac{\delta^{2}}{\delta^{\frac{n-2}{2}}}\left(O(|x|^{2})\partial_{ij}v_{q}+O(|x|)\partial_{j}v_{q}\right)\\
 & =\frac{1}{\delta^{\frac{n-2}{2}}}\left(-\frac{1}{3}\left(\bar{R}_{ikjl}x_{k}x_{l}+R_{ninj}x_{n}^{2}\right)\partial_{ij}^{2}\left(U(x)\chi(\delta x)\right)+\delta O(|x|^{3})\right)
\end{align*}
Thus, this term cancels the first order term in (\ref{eq:DW}) and
we get 
\begin{equation}
\|\Delta_{\tilde{g}_{q}}(W_{\delta,q}+\delta^{2}V_{\delta,q})\|_{L^{\frac{2n}{n+2}}(M)}=O\left(\delta^{n\frac{n+2}{2n}}\frac{1}{\delta^{\frac{n-2}{2}}}\delta\right)=O(\delta^{3})\label{eq:deltaw+v}
\end{equation}

Also we have
\begin{align}
\int_{M}R_{\tilde{g}_{q}}(W_{\delta,q}+\delta^{2}V_{\delta,q})\Lambda_{q}^{-1}Rd\mu_{\tilde{g}} & \le c\|R_{\tilde{g}_{q}}(W_{\delta,q}+\delta^{2}V_{\delta,q})\|_{L^{\frac{2n}{n+2}}(M,\tilde{g}_{q})}\|\Lambda_{q}^{-1}R\|_{L^{\frac{2n}{n-2}}(M,\tilde{g}_{q})}\nonumber \\
 & \le o(\delta^{3})\|R\|_{g}\label{eq:L2}
\end{align}
In fact, since $R_{\tilde{g}_{q}}(q)=R_{\tilde{g}_{q},i}(q)=R_{\tilde{g}_{q},t}(q)=0$
(see \cite{M1}), by the decay
of $U$ and since $|v_{q}(y)|\le C|y|^{4-n}$ we have, in local coordinates
\begin{multline*}
\|R_{\tilde{g}_{q}}(W_{\delta,q}+\delta^{2}V_{\delta,q})\|_{L^{\frac{2n}{n+2}}(M,\tilde{g}_{q})}\le O(\delta^{2})\left\{ \int_{\mathbb{R}_{+}^{n}}\left[R_{\tilde{g}_{q}}(\delta y)U(y)\chi(\delta y)\right]^{\frac{2n}{n+2}}dy\right\} ^{\frac{n+2}{2n}}\\
+O(\delta^{4})\left\{ \int_{\mathbb{R}_{+}^{n}}\left[R_{\tilde{g}_{q}}(\delta y)v_{q}(y)\chi(\delta y)\right]^{\frac{2n}{n+2}}dy\right\} ^{\frac{n+2}{2n}}\\
=o(\delta^{3})\left\{ \int_{\mathbb{R}_{+}^{n}}\left[y^{2}U(y)\right]^{\frac{2n}{n+2}}dy\right\} ^{\frac{n+2}{2n}}+O(\delta^{4})\left\{ \int_{\mathbb{R}_{+}^{n}}\left[v_{q}(y)\right]^{\frac{2n}{n+2}}dy\right\} ^{\frac{n+2}{2n}}+o(\delta^{3})\\
=o(\delta^{3})\text{ since }n>10.
\end{multline*}
For the boundary term we have
\begin{multline*}
B_{\tilde{g}_{q}}(\Lambda_{q}^{-1}\Gamma)-B_{\tilde{g}_{q}}(W_{\delta,q}+\delta^{2}V_{\delta,q})\\
=\left\{ (n-2)\left((W_{\delta,q}+\delta^{2}V_{\delta,q})^{+}\right)^{\frac{n}{n-2}}-(n-2)\left(W_{\delta,q}^{+}\right)^{\frac{n}{n-2}}-\delta^{2}\frac{\partial V_{\delta,q}}{\partial\nu}\right\} \\
+\left\{ (n-2)\left(W_{\delta,q}^{+}\right)^{\frac{n}{n-2}}-\frac{\partial W_{\delta,q}}{\partial\nu}\right\} -\frac{n-2}{2}h_{\tilde{g}}(W_{\delta,q}+\delta^{2}V_{\delta,q})
\end{multline*}
Since the boundary is umbilic, we have $h_{\tilde{g}_{q}}(q)=h_{\tilde{g}_{q},i}(q)=h_{\tilde{g}_{q},ik}(q)=0$
so we estimate
\begin{align}
\int_{\partial M}h_{\tilde{g}_{q}}(W_{\delta,q}+\delta^{2}V_{\delta,q})\Lambda_{q}^{-1}Rd\sigma_{\tilde{g}_{q}} & \le c\|h_{\tilde{g}_{q}}(W_{\delta,q}+\delta^{2}V_{\delta,q})\|_{L^{\frac{2(n-1)}{n+2}}(M,\tilde{g}_{q})}\|\Lambda_{q}^{-1}R\|_{\tilde{g}_{q}}\nonumber \\
 & =O(\delta^{4})\|R\|_{g}\label{eq:hresto1}
\end{align}
 Indeed, as in (\ref{eq:L2})
\begin{multline}
\|h_{\tilde{g}_{q}}(W_{\delta,q}+\delta^{2}V_{\delta,q})\|_{L^{\frac{2(n-1)}{n+2}}(\partial M,\tilde{g}_{q})}\le\|h_{\tilde{g}_{q}}W_{\delta,q}\|_{L^{\frac{2(n-1)}{n+2}}(\partial M,\tilde{g}_{q})}+\delta^{2}\|V_{\delta,q}\|_{L^{\frac{2(n-1)}{n+2}}(\partial M,\tilde{g}_{q})}\\
=\left\{ \int_{\mathbb{R}^{n-1}}\left[h_{\tilde{g}_{q}}(0,\delta z)U(0,z)\chi(0,\delta z)\right]^{\frac{2(n-1)}{n+2}}dz\right\} ^{\frac{n+2}{2(n-1)}}\\
+\delta^{2}\left\{ \int_{\mathbb{R}^{n-1}}\left[h_{\tilde{g}_{q}}(0,\delta z)v_{q}(0,z)\chi(0,\delta z)\right]^{\frac{2(n-1)}{n+2}}dz\right\} ^{\frac{n+2}{2(n-1)}}\\
=o(\delta^{3})\left\{ \int_{\mathbb{R}^{n-1}}\left[z^{2}U(0,z)\right]^{\frac{2(n-1)}{n+2}}dz\right\} ^{\frac{n+2}{2(n-1)}}+O(\delta^{4})\left\Vert v_{q}(0,z)\right\Vert _{L^{\frac{2(n-1)}{n+2}}(\mathbb{R}^{n-1})}+o(\delta^{3})\\
=O(\delta^{4})\text{ since }n>10.\label{eq:hresto2}
\end{multline}
Here we considered $y=(t,z)$ with $t>0$ and $z\in\mathbb{R}^{n-1}$.

Easily we get
\[
\left\Vert (n-2)W_{\delta,q}^{\frac{n}{n-2}}-\frac{\partial}{\partial\nu}W_{\delta,q}\right\Vert _{L^{\frac{2(n-1)}{n}}(\partial M)}=O(\delta^{3})
\]
since $U$ solves (\ref{eq:Udelta}).

For the last term we estimate 
\begin{multline*}
\int_{\partial M}\left\{ (n-2)\left[\left((W_{\delta,q}+\delta^{2}V_{\delta,q})^{+}\right)^{\frac{n}{n-2}}
-W_{\delta,q}^{\frac{n}{n-2}}\right]-\delta^{2}\frac{\partial V_{\delta,q}}{\partial\nu}\right\}
 \Lambda_q^{-1}Rd\sigma_{\tilde{g}_{q}}\\
\le c\left\Vert (n-2)\left[\left((W_{\delta,q}+\delta^{2}V_{\delta,q})^{+}\right)^{\frac{n}{n-2}}
-W_{\delta,q}^{\frac{n}{n-2}}\right]-\delta^{2}
\frac{\partial V_{\delta,q}}{\partial\nu}\right\Vert _{L^{\frac{2(n-1)}{n}}(\partial M,\tilde{g}_{q})}\|R\|_{g}
\end{multline*}
and, by Taylor expansion and by definition of the function $v_{q}$
(see (\ref{eq:vqdef}) ) 
\begin{multline*}
\left\Vert (n-2)\left[\left((W_{\delta,q}+\delta^{2}V_{\delta,q})^{+}\right)^{\frac{n}{n-2}}-W_{\delta,q}^{\frac{n}{n-2}}\right]-\delta^{2}\frac{\partial V_{\delta,q}}{\partial\nu}\right\Vert _{L^{\frac{2(n-1)}{n}}(\partial M,\tilde{g}_{q})}\\
\le\left\Vert (n-2)\left[\left((U+\delta^{2}v_{q})^{+}\right)^{\frac{n}{n-2}}-U^{\frac{n}{n-2}}\right]+\delta^{2}\frac{\partial v_{q}}{\partial t}\right\Vert _{L^{\frac{2(n-1)}{n}}(\partial\mathbb{R}_{+}^{n})}+o(\delta^{3})\\
\le\delta^{2}\left\Vert n\left((U+\theta\delta^{2}v_{q})^{+}\right)^{\frac{2}{n-2}}v_{q}+\frac{\partial v_{q}}{\partial t}\right\Vert _{L^{\frac{2(n-1)}{n}}(\partial\mathbb{R}_{+}^{n})}+o(\delta^{3})\\
=\delta^{2}n\left\Vert \left((U+\theta\delta^{2}v_{q})^{+}\right)^{\frac{2}{n-2}}v_{q}-U^{\frac{2}{n-2}}v_{q}\right\Vert _{L^{\frac{2(n-1)}{n}}(\partial\mathbb{R}_{+}^{n})}+o(\delta^{3}).
\end{multline*}
We observe that, chosen a large positive $R$, we have $U+\theta\delta^{2}v_{q}>0$
in $B(0,R)$ for some $\delta$. Moreover, on the complementary of
this ball, we have $\frac{c}{|y|^{n-2}}\le U(y)\le\frac{C}{|y|^{n-2}}$
and $|v_{q}|\le\frac{C_{1}}{|y|^{n-4}}$ for some positive constants
$c,C,C_{1}$. So it is possible to prove that, for $\delta$ small
enough, $U+\theta\delta^{2}v_{q}>0$ if $|y|\le1/\delta$. At this
point 
\begin{multline*}
\int_{\partial\mathbb{R}_{+}^{n}}\left[\left|\left((U+\theta\delta^{2}v_{q})^{+}\right)^{\frac{2}{n-2}}-U^{\frac{2}{n-2}}\right||v_{q}|\right]^{\frac{2(n-1)}{n}}\\
=\int_{U+\theta\delta v_{q}>0}\left[\left|\left((U+\theta\delta^{2}v_{q})^{+}\right)^{\frac{2}{n-2}}-U^{\frac{2}{n-2}}\right||v_{q}|\right]^{\frac{2(n-1)}{n}}dz\\
+\int_{U+\theta\delta v_{q}\le0}\left[\left|\left((U+\theta\delta^{2}v_{q})^{+}\right)^{\frac{2}{n-2}}-U^{\frac{2}{n-2}}\right||v_{q}|\right]^{\frac{2(n-1)}{n}}dz\\
=\delta^{\frac{4(n-1)}{n}}\int_{U+\theta\delta v_{q}>0}\left(U+\theta_{1}\delta^{2}v_{q}\right)^{\frac{-2(n-1)(n-4)}{n(n-2)}}|v_{q}|^{\frac{4(n-1)}{n}}dz\\
+\int_{U+\theta\delta v_{q}\le0}U^{\frac{4(n-1)}{n(n-2)}}|v_{q}|^{\frac{2(n-1)}{n}}dz\\
\le\delta^{\frac{4(n-1)}{n}}\int_{U+\theta\delta v_{q}>0}\left(U+\theta_{1}\delta^{2}v_{q}\right)^{\frac{-2(n-1)(n-4)}{n(n-2)}}|v_{q}|^{\frac{4(n-1)}{n}}dz\\
+\int_{|z|>\frac{1}{\delta}}U^{\frac{4(n-1)}{n(n-2)}}|v_{q}|^{\frac{2(n-1)}{n}}dz
\end{multline*}
and, since $n>10$ one can check that $\int_{U+\theta\delta v_{q}>0}\left(U+\theta_{1}\delta^{2}v_{q}\right)^{\frac{-2(n-1)(n-4)}{n(n-2)}}|v_{q}|^{\frac{4(n-1)}{n}}dz$ 
is bounded and that 
\begin{align*}
\int_{|z|>\frac{1}{\delta}}U^{\frac{4(n-1)}{n(n-2)}}|v_{q}|^{\frac{2(n-1)}{n}}dz & \le C\int_{|z|>\frac{1}{\delta}}\frac{1}{|z|^{\frac{4(n-1)}{n}}}\frac{1}{|z|^{\frac{2(n-1)(n-4)}{n}}}dz\\
 & \le C\int_{\frac{1}{\delta}}^{\infty}r^{-\frac{(n-2)^{2}}{n}}=O(\delta^{\frac{(n-2)^{2}-1}{n}})=o(\delta^{\frac{4(n-1)}{n}})
\end{align*}
thus $\left\Vert (n-2)\left[\left((W_{\delta,q}+\delta^{2}V_{\delta,q})^{+}\right)^{\frac{n}{n-2}}-W_{\delta,q}^{\frac{n}{n-2}}\right]-\delta^{2}\frac{\partial V_{\delta,q}}{\partial\nu}\right\Vert _{L^{\frac{2(n-1)}{n}}(\partial M,\tilde{g}_{q})}=O(\delta^{4})$
and 
\begin{multline*}
\int_{\partial M}\left\{ (n-2)\left[\left((W_{\delta,q}+\delta^{2}V_{\delta,q})^{+}\right)^{\frac{n}{n-2}}-W_{\delta,q}^{\frac{n}{n-2}}\right]-\delta\frac{\partial V_{\delta,q}}{\partial\nu}\right\}\Lambda_q^{-1} Rd\sigma_{\tilde{g}_{q}}\\
\le O(\delta^{4})\|R\|_{g}.
\end{multline*}
\end{proof}
At this point we can can use the same strategy of proposition 11 of
\cite{GMP} to prove the following result
\begin{prop}
There exists a positive constant $C$ such that for $\varepsilon,\delta$
small, for any $q\in\partial M$ there exists a unique $\tilde{\Phi}=\tilde{\Phi}_{\varepsilon,\delta,q}\in\tilde{K}_{\delta,q}^{\bot}$
which solves (\ref{eq:Pibot}) such that 
\[
\|\tilde{\Phi}\|_{g}=\|\Lambda_{q}\Phi\|_{g}\le C(\varepsilon\delta+\delta^{3}).
\]
\end{prop}

\section{The reduced functional}\label{reduced}

In this section we perform the expansion of the functional with respect
to the parameter $\varepsilon,\delta$.
\begin{lem}
\label{lem:JWpiuPhi}Assume $n\ge10$. It holds 
\[
J_{\varepsilon,g}(\tilde{W}_{\delta,q}+\delta^{2}\tilde{V}_{\delta,q}+\tilde{\Phi})-J_{\varepsilon,g}(\tilde{W}_{\delta,q}+\delta^{2}\tilde{V}_{\delta,q})=O\left((\delta^{2}+\varepsilon\delta)\|\tilde{\Phi}\|_{g}+\|\tilde{\Phi}\|_{g}^{2}\right)
\]
$C^{0}$-uniformly for $q\in\partial M$.
\end{lem}
The proof of this Lemma is postponed to the appendix. 

We recall here an useful result contained in \cite{M1}.
\begin{rem}
\label{rem:I}It holds
\begin{align*}
I_{1}:= & \int_{\mathbb{R}_{+}^{n}}\frac{t^{2}}{\left((1+t)^{2}+|z|^{2}\right)^{n-2}}dtdz=\frac{4(n-2)}{n+1}I_{2}\\
I_{2}:= & \int_{\mathbb{R}_{+}^{n}}\frac{t^{2}|z|^{4}}{\left((1+t)^{2}+|z|^{2}\right)^{n}}dtdz\\
I_{3}:= & \int_{\mathbb{R}_{+}^{n}}\frac{t^{4}|z|^{2}}{\left((1+t)^{2}+|z|^{2}\right)^{n}}dtdz=\frac{12}{(n-2)(n+1)}I_{2}\\
I_{4}:= & \int_{\mathbb{R}_{+}^{n}}\frac{ |z|^{2}}{\left((1+t)^{2}+|z|^{2}\right)^{n-2}}dtdz 
\end{align*}
and 
\begin{equation}
\int_{\mathbb{R}_{+}^{n}}\frac{t^{2}z_{i}^{4}}{\left((1+t)^{2}+|z|^{2}\right)^{n}}dtdz=3\int_{\mathbb{R}_{+}^{n}}\frac{t^{2}z_{i}^{2}z_{j}^{2}}{\left((1+t)^{2}+|z|^{2}\right)^{n}}dtdz=\frac{3}{n^{2}-1}I_{2}.\label{eq:t2z4}
\end{equation}
\end{rem}
\begin{lem}
\label{lem:espansione}It holds
\begin{equation}\label{1claim}
J_{\varepsilon,g}(\tilde{W}_{\delta,q}+\delta^{2}\tilde{V}_{\delta,q})=A+\varepsilon\delta\gamma(q)B+\delta^{4}\varphi(q)+O(\varepsilon\delta^{3})+O(\delta^{5})
\end{equation}
where 
$$\varphi(q)=  \frac{1}{2}\int_{\mathbb{R}_{+}^{n}}v_{q}\Delta v_{q}dtdz+\frac{(n-2)(n-8)}{4(n^{2}-1)}R_{nn,nn} (q)I_{2}-\frac{n-2}{96(n-1)}|\bar{W}(q)|^{2}I_4,$$
 the constants
$I_2,$   $I_4$ are defined in Remark \ref{rem:I} and
$$A=   \frac{1}{2}\int_{\mathbb{R}_{+}^{n}}|\nabla U(t,z)|^{2}dtdz-\frac{(n-2)^{2}}{2(n-1)}\int_{\mathbb{R}^{n-1}}U(0,z)^{\frac{2(n-1)}{(n-2)}}dz,\ B=   \frac{1}{2}\int_{\mathbb{R}^{n-1}}U(0,z)^{2}dz 
$$
Here $\bar{W}(q)$ is the Weyl tensor restricted to boundary and we
consider the local coordinates $y=(t,z)$ with $t>0$ and $z\in\mathbb{R}^{n-1}$. 

Moreover
\begin{equation}\label{sign}\varphi(q)\le 0\ \hbox{for any}\ q\in\partial M.\end{equation}
\end{lem}
\begin{proof}
First of all, we point out that the last claim \eqref{sign} immediately follows by \eqref{new} and by the identity $R_{nn,nn}=-2R_{nins}^{2}$ (see \cite[Prop. 3.2 (7)]{M1}).

Now, let us prove that \eqref{1claim} holds.
We have 
\begin{align*}
J_{\varepsilon,g}(\tilde{W}_{\delta,q}+\delta^{2}\tilde{V}_{\delta,q})= & \frac{1}{2}\int_{M}|\nabla_{\tilde{g}_{q}}(W_{\delta,q}+\delta^{2}V_{\delta,q})|^{2}d\mu_{\tilde{g}_{q}}+\frac{n-2}{8(n-1)}\int_{M}R_{\tilde{g}_{q}}(W_{\delta,q}+\delta^{2}V_{\delta,q})^{2}d\mu_{\tilde{g}_{q}}\\
 & +\frac{1}{2}\varepsilon\int_{\partial M}\Lambda_{q}^{-\frac{2}{n-2}}\gamma(W_{\delta,q}+\delta^{2}V_{\delta,q})^{2}d\sigma_{\tilde{g}_{q}}\\
 & -\frac{(n-2)^{2}}{2(n-1)}\int_{\partial M}\left[\left((W_{\delta,q}+\delta^{2}V_{\delta,q})^{+}\right)^{\frac{2(n-1)}{n-2}}-W_{\delta,q}^{\frac{2(n-1)}{n-2}}\right]d\sigma_{\tilde{g}_{q}}\\
 & -\frac{(n-2)^{2}}{2(n-1)}\int_{\partial M}W_{\delta,q}^{\frac{2(n-1)}{n-2}}d\sigma_{\tilde{g}_{q}}+\frac{n-2}{4}\int_{\partial M}h_{\tilde{g}_{q}}(W_{\delta,q}+\delta^{2}V_{\delta,q})^{2}d\sigma_{\tilde{g}_{q}}\\
=: & A_{1}+A_{2}+A_{3}+A_{4}+A_{5}+A_{6}.
\end{align*}
We use the change of variables $y=\delta x=(\delta t,\delta z)\in\mathbb{R}_{+}^{n}$
with $t\ge0$ and $z\in\mathbb{R}^{n-1}.$ On $\partial\mathbb{R}_{+}^{n}$
also we use $y=(0,\zeta)$
\begin{align*}
A_{6} & =\frac{n-2}{4}\int_{\partial M}h_{\tilde{g}_{q}}(W_{\delta,q}+\delta^{2}V_{\delta,q})^{2}d\sigma_{\tilde{g}_{q}}\\
 & =\frac{n-2}{4}\int_{\mathbb{R}^{n-1}}\frac{1}{\delta^{n-2}}h_{\tilde{g}_{q}}(0,\zeta)\left(U\left(0,\frac{\zeta}{\delta}\right)+\delta^{2}v_{q}\left(0,\frac{\zeta}{\delta}\right)\right)^{2}\chi^{2}(0,\zeta)|\tilde{g}_{q}|^{\frac{1}{2}}(0,\zeta)dy\\
 & =\frac{n-2}{4}\delta\int_{\mathbb{R}^{n-1}}h_{\tilde{g}_{q}}(0,\delta z)\left(U\left(0,z\right)+\delta^{2}v_{q}\left(0,z\right)\right)^{2}\chi^{2}(0,\delta z)|\tilde{g}_{q}|^{\frac{1}{2}}(0,\delta z)dz.
\end{align*}
In light of (\ref{eq:hij}), by Taylor expansion of $h_{\tilde{g}_{q}}(0,\delta z)$,
since by symmetry the first term is zero and $n>10$, we have
\[
A_{6}=\frac{n-2}{4}\delta^{4}\int_{\mathbb{R}^{n-1}}\partial_{ijk}h_{\tilde{g}_{q}}(q)z_{i}z_{j}z_{k}U\left(0,z\right)^{2}dz+O(\delta^{5})=O(\delta^{5})
\]
In a similar way, expanding $R_{\tilde{g}_{q}}$ we get, by (\ref{eq:Rii}),
\begin{align*}
A_{2} & =\frac{n-2}{8(n-1)}\int_{\mathbb{R}_{+}^{n}}\delta^{2}R_{\tilde{g}_{q}}(\delta x)(U^{2}+\delta^{2}Uv_{q}+\delta^{4}v_{q}^{2})\chi^{2}(\delta x)|\tilde{g}_{q}|^{\frac{1}{2}}(\delta x)dx\\
 & =\frac{n-2}{16(n-1)}\delta^{4}\int_{\mathbb{R}_{+}^{n}}\partial_{ab}^{2}R_{\tilde{g}_{q}}(q)x_{a}x_{b}U^{2}dx+O(\delta^{5})
\end{align*}
By symmetry reasons, and recalling (\ref{eq:Rii}), we have
\begin{align}
A_{2} & =\delta^{4}\frac{n-2}{16(n-1)}\left[\partial_{z_{i}z_{i}}^{2}R_{\tilde{g}_{q}}(q)\int_{\mathbb{R}_{+}^{n}}\frac{|z|^{2}U^{2}(z,t)}{n-1}dzdt+\partial_{tt}^{2}R_{\tilde{g}_{q}}(q)\int_{\mathbb{R}_{+}^{n}}t^{2}U^{2}(z,t)dzdt\right]+O(\delta^{5})\nonumber \\
 & =-\delta^{4}\frac{n-2}{96(n-1)^{2}}|\bar{W}(q)|^{2}\int_{\mathbb{R}_{+}^{n}}|z|^{2}U^{2}(z,t)dzdt\label{eq:A2}\\
 & +\delta^{4}\frac{n-2}{16(n-1)}\partial_{tt}^{2}R_{\tilde{g}_{q}}\int_{\mathbb{R}_{+}^{n}}t^{2}U^{2}(z,t)dzdt+O(\delta^{5}).\nonumber 
\end{align}
Analogously we have, since $\Lambda_{q}(q)=1$,
\begin{align*}
A_{3} & =\frac{1}{2}\varepsilon\delta\int_{\mathbb{R}^{n-1}}\Lambda_{q}^{-\frac{2}{n-2}}\gamma U^{2}(0,z)dz+O(\varepsilon\delta^{3})\\
 & =\frac{1}{2}\varepsilon\delta\gamma(q)\int_{\mathbb{R}^{n-1}}U^{2}(0,z)dz+O(\varepsilon\delta^{3}).
\end{align*}
Also, by (\ref{eq:|g|}) 
\begin{align*}
A_{5} & =-\frac{(n-2)^{2}}{2(n-1)}\int_{\mathbb{R}^{n-1}}\left[U(0,z)\chi(0,\delta z)\right]^{\frac{2(n-1)}{n-2}}|\tilde{g}_{q}(0,\delta z)|^{\frac{1}{2}}dz\\
 & =-\frac{(n-2)^{2}}{2(n-1)}\int_{\mathbb{R}^{n-1}}U(0,z)^{\frac{2(n-1)}{n-2}}dz+O(\delta^{5}).
\end{align*}
For $A_{4}$, expanding twice by Taylor formula, and since $\int_{\mathbb{R}^{n-1}}U(0,z)^{\frac{2}{n-2}}v_{q}(0,z)dz=0$,
we have
\begin{align}
A_{4}= & -(n-2)\delta^{2}\int_{\mathbb{R}^{n-1}}\left[U(0,z)^{\frac{n}{n-2}}v_{q}\right]dz\nonumber \\
 & -\frac{n}{2}\delta^{4}\int_{\mathbb{R}^{n-1}}\left[\left((U(0,z)+\theta\delta^{2}v_{q})^{+}\right)^{\frac{2}{n-2}}v_{q}^{2}\right]dz+o(\delta^{4})\nonumber \\
= & -\frac{n}{2}\delta^{4}\int_{\mathbb{R}^{n-1}}\left[U(0,z)^{\frac{2}{n-2}}v_{q}^{2}\right]dz+o(\delta^{4}).\label{eq:A4finale}
\end{align}
Concerning the gradient term we have 
\begin{align*}
A_{1} & =\frac{1}{2}\int_{M}|\nabla_{\tilde{g}_{q}}W_{\delta,q}|^{2}d\mu_{\tilde{g}_{q}}+\delta^{2}\int_{M}\nabla_{\tilde{g}_{q}}W_{\delta,q}\cdot\nabla V_{\delta,q}d\mu_{\tilde{g}_{q}}+\frac{\delta^{4}}{2}\int_{M}|\nabla_{\tilde{g}_{q}}V_{\delta,q}|^{2}d\mu_{\tilde{g}_{q}}\\
 & =:L_{1}+L_{2}+L_{3}.
\end{align*}
We have, by (\ref{eq:|g|}) and (\ref{eq:gij}) and integrating by
parts
\begin{align}
L_{3} & =\frac{\delta^{4}}{2}\int_{\mathbb{R}_{+}^{n}}\left[\tilde{g}_{q}^{ij}(\delta x)\partial_{y_{i}}(v_{q}(x)\chi(\delta x))\partial_{x_{j}}(v_{q}(x)\chi(\delta x))+(\partial_{x_{n}}(v_{q}(x)\chi(\delta x)))^{2}\right]|\tilde{g}_{q}(\delta x)|^{\frac{1}{2}}dy\nonumber \\
 & =\frac{\delta^{4}}{2}\int_{\mathbb{R}_{+}^{n}}|\nabla v_{q}|^{2}dx+O(\delta^{5}).\nonumber \\
 & =-\frac{\delta^{4}}{2}\int_{\mathbb{R}_{+}^{n}}v_{q}\Delta v_{q}dzdt+\frac{\delta^{4}}{2}\int_{\mathbb{R}^{n-1}}v_{q}\frac{\partial}{\partial\nu}v_{q}dz+O(\delta^{5})\nonumber \\
 & =-\frac{\delta^{4}}{2}\int_{\mathbb{R}_{+}^{n}}v_{q}\Delta v_{q}dzdt+\delta^{4}\frac{n}{2}\int_{\mathbb{R}^{n-1}}U^{\frac{2}{n-2}}v_{q}^{2}dz+O(\delta^{5})\label{eq:L3finale}
\end{align}
In light of (\ref{eq:Udelta}) and (\ref{eq:Uvq}) we have
\begin{align*}
\int_{\mathbb{R}^{n}}\nabla U\nabla v_{q}dzdt & =-\int_{\mathbb{R}_{+}^{n}}\Delta Uv_{q}dzdt+\int_{\mathbb{R}_{+}^{n-1}}\frac{\partial}{\partial\nu}Uv_{q}dzdt\\
 & =-(n-2)\int_{\mathbb{R}_{+}^{n-1}}U^{\frac{n}{n-2}}v_{q}dzdt=0
\end{align*}
So we have
\begin{align}
L_{2}= & \delta^{2}\int_{\mathbb{R}_{+}^{n}}\nabla U\nabla v_{q}dzdt\label{eq:L2parziale}\\
 & +\delta^{4}\int_{\mathbb{R}_{+}^{n}}\left(\frac{1}{3}\bar{R}_{ikjl}z_{k}z_{l}\partial_{z_{i}}U\partial_{z_{j}}v_{q}+R_{ninj}t^{2}\partial_{z_{i}}U\partial_{z_{j}}v_{q}\right)dzdt+O(\delta^{5})\nonumber \\
 & =\delta^{4}\int_{\mathbb{R}_{+}^{n}}\left(\frac{1}{3}\bar{R}_{ikjl}z_{k}z_{l}\partial_{z_{i}}U\partial_{z_{j}}v_{q}+R_{ninj}t^{2}\partial_{z_{i}}U\partial_{z_{j}}v_{q}\right)dzdt+O(\delta^{5}).\nonumber 
\end{align}
At this point we can calculate $A_{4}+L_{2}+L_{3}$. By (\ref{eq:A4finale}),
(\ref{eq:L3finale}), (\ref{eq:L2parziale}) and by definition of
$v_{q}$ (\ref{eq:vqdef}) we get, integrating by parts
\begin{align}
A_{4}+L_{2}+L_{3}= & -\frac{\delta^{4}}{2}\int_{\mathbb{R}_{+}^{n}}v_{q}\Delta v_{q}dzdt+\delta^{4}\int_{\mathbb{R}_{+}^{n}}\left(\frac{1}{3}\bar{R}_{ikjl}z_{k}z_{l}+R_{ninj}t^{2}\right)\partial_{z_{i}}U\partial_{z_{j}}v_{q}dzdt+O(\delta^{5})\nonumber \\
= & -\frac{\delta^{4}}{2}\int_{\mathbb{R}_{+}^{n}}v_{q}\Delta v_{q}dzdt-\delta^{4}\int_{\mathbb{R}_{+}^{n}}\partial_{z_{j}}\left(\frac{1}{3}\bar{R}_{ikjl}z_{k}z_{l}+R_{ninj}t^{2}\right)\partial_{z_{i}}Uv_{q}dzdt\nonumber \\
 & -\delta^{4}\int_{\mathbb{R}_{+}^{n}}\left(\frac{1}{3}\bar{R}_{ikjl}z_{k}z_{l}+R_{ninj}t^{2}\right)\partial_{z_{j}}\partial_{z_{i}}Uv_{q}dzdt\nonumber \\
 & +\delta^{4}\int_{\partial\mathbb{R}_{+}^{n}}\left(\frac{1}{3}\bar{R}_{ikjl}z_{k}z_{l}+R_{ninj}t^{2}\right)v_{q}\partial_{z_{i}}U\nu_{j}dzdt+O(\delta^{5})\nonumber \\
= & \frac{\delta^{4}}{2}\int_{\mathbb{R}_{+}^{n}}v_{q}\Delta v_{q}dzdt+O(\delta^{5})\label{eq:sommaparziale}
\end{align}
since $\nu_{j}=0$ for all $j=1,\dots,n-1$, and by the symmetries
of the curvature tensor and by (\ref{eq:Ricci}) we have
\begin{multline*}
\int_{\mathbb{R}_{+}^{n}}\partial_{z_{j}}\left(\frac{1}{3}\bar{R}_{ikjl}z_{k}z_{l}+R_{ninj}t^{2}\right)\partial_{z_{i}}Uv_{q}dzdt\\
=\frac{1}{3}\bar{R}_{ikjl}\int_{\mathbb{R}_{+}^{n}}\partial_{z_{j}}(z_{k}z_{l})\partial_{z_{i}}Uv_{q}dzdt\\
=\frac{1}{3}\bar{R}_{il}\int_{\mathbb{R}_{+}^{n}}z_{l}\partial_{z_{i}}Uv_{q}dzdt+\frac{1}{3}\bar{R}_{ikjj}\int_{\mathbb{R}_{+}^{n}}z_{k}\partial_{z_{i}}Uv_{q}dzdt=0.
\end{multline*}
 Finally we have, by (\ref{eq:|g|}) and (\ref{eq:gij}) and since
the terms of odd degree disappear by symmetry

\begin{align}
L_{1}= & \frac{1}{2}\int_{\mathbb{R}_{+}^{n}}\tilde{g}_{ij}(\delta x)(\partial_{x_{i}}(U\chi),\partial_{x_{j}}(U\chi))+\partial_{x_{n}}(U\chi)|\tilde{g}_{q}|^{\frac{1}{2}}dx\nonumber \\
 & =\frac{1}{2}\int_{\mathbb{R}_{+}^{n}}|\nabla U|^{2}dzdt+\delta^{2}\int_{\mathbb{R}^{n}}\left(\frac{1}{3}\bar{R}_{ikjl}z_{k}z_{l}+R_{ninj}t^{2}\right)\partial_{z_{i}}U\partial_{z_{j}}Udzdt\nonumber \\
 & +\frac{\delta^{4}}{2}\int_{\mathbb{R}_{+}^{n}}\left(\frac{1}{20}\bar{R}_{ikjl,mp}+\frac{1}{15}\bar{R}_{iksl}\bar{R}_{jmsp}\right)z_{k}z_{l}z_{m}z_{p}\partial_{z_{i}}U\partial_{z_{j}}Udzdt\nonumber \\
 & +\frac{\delta^{4}}{2}\int_{\mathbb{R}_{+}^{n}}\left(\frac{1}{2}R_{ninj,kl}+\frac{1}{3}\text{Sym}_{ij}(\bar{R}_{iksl}R_{nsnj})\right)t^{2}z_{k}z_{l}\partial_{z_{i}}U\partial_{z_{j}}Udzdt\nonumber \\
 & +\frac{\delta^{4}}{2}\int_{\mathbb{R}_{+}^{n}}\left(\frac{1}{3}R_{ninj,nk}t^{3}z_{k}+\frac{1}{12}\left(R_{ninj,nn}+8R_{nins}R_{nsnj}\right)t^{4}\right)\partial_{z_{i}}U\partial_{z_{j}}Udzdt\nonumber \\
 & +O(\delta^{5}).\label{eq:L1inizio}
\end{align}
Now we prove that all the terms of order $\delta^{2}$ vanish. Since
$\partial_{z_{i}}U=(2-n)\frac{z_{i}}{\left((1+t)^{2}+|z|^{2}\right)^{\frac{n}{2}}}$
we have
\begin{multline*}
\int_{\mathbb{R}_{+}^{n}}\left(\frac{1}{3}\bar{R}_{ikjl}z_{k}z_{l}+R_{ninj}t^{2}\right)\partial_{z_{i}}U\partial_{z_{j}}Udzdt\\
=\frac{(n-2)^{2}}{3}\int_{\mathbb{R}_{+}^{n}}\bar{R}_{ikjl}\frac{z_{k}z_{l}z_{i}z_{j}}{\left((1+t)^{2}+|z|^{2}\right)^{n}}dzdt+(n-2)^{2}\int_{\mathbb{R}_{+}^{n}}R_{ninj}\frac{t^{2}z_{i}z_{j}}{\left((1+t)^{2}+|z|^{2}\right)^{n}}dzdt.
\end{multline*}
By symmetry reasons and by (\ref{eq:Ricci}) we have 
\[
\int_{\mathbb{R}_{+}^{n}}R_{ninj}\frac{t^{2}z_{i}z_{j}}{\left((1+t)^{2}+|z|^{2}\right)^{n}}dzdt=\frac{1}{n-1}R_{nn}\int_{\mathbb{R}_{+}^{n}}\frac{t^{2}|z|^{2}}{\left((1+t)^{2}+|z|^{2}\right)^{n}}dzdt=0.
\]
Moreover, by symmetry the integrals $\int_{\mathbb{R}^{n}}\frac{z_{k}z_{l}z_{i}z_{j}}{\left((1+t)^{2}+|z|^{2}\right)^{n}}dzdt$
are non zero only when $i=j=k=l$, $i=j\neq k=l$, $i=k\neq j=l$
and $i=l\neq j=k$. Since $\bar{R}_{iijj}=0$ for all $i,j$ we get
\begin{multline*}
\int_{\mathbb{R}_{+}^{n}}\bar{R}_{ikjl}\frac{z_{k}z_{l}z_{i}z_{k}}{\left((1+t)^{2}+|z|^{2}\right)^{n}}dzdt\\
=\sum_{i\neq k}\int_{\mathbb{R}_{+}^{n}}\bar{R}_{ikik}\frac{z_{i}^{2}z_{k}^{2}}{\left((1+t)^{2}+|z|^{2}\right)^{n}}dzdt+\sum_{i\neq k}\int_{\mathbb{R}_{+}^{n}}\bar{R}_{ikki}\frac{z_{i}^{2}z_{k}^{2}}{\left((1+t)^{2}+|z|^{2}\right)^{n}}dzdt=0.
\end{multline*}
By the symmetries of the curvature tensor (see \cite[page 1614, formula C]{M1})
we get 
\begin{equation}
G_{1}:=\int_{\mathbb{R}_{+}^{n}}\left(\frac{1}{20}\bar{R}_{ikjl,mp}+\frac{1}{15}\bar{R}_{iksl}\bar{R}_{jmsp}\right)z_{k}z_{l}z_{m}z_{p}\partial_{z_{i}}U\partial_{z_{j}}Udzdt=0\label{eq:G1}
\end{equation}
Moreover, using that $R_{nn,nn}=-2R_{nins}^{2}$,
we get 
\begin{align}
G_{3} & :=\sum_{i}\frac{1}{12}\left(R_{nini,nn}+8R_{nins}R_{nsni}\right)\int_{\mathbb{R}_{+}^{n}}t^{4}|\partial_{z_{i}}U|^{2}dzdt\nonumber \\
 & =\frac{(n-2)^{2}}{12(n-1)}\left(R_{nn,nn}+8R_{nins}^{2}\right)\int_{\mathbb{R}_{+}^{n}}\frac{t^{4}|z|^{2}}{\left((1+t)^{2}+|z|^{2}\right)^{n}}dzdt\nonumber \\
 & =\frac{n-2}{n^{2}-1}\left(R_{nn,nn}+8R_{nins}^{2}\right)I_{2}=\frac{6(n-2)}{n^{2}-1}R_{nins}^{2}I_{2}.\label{eq:G3}
\end{align}
It remains
\[
G_{2}:=(n-2)^{2}\int_{\mathbb{R}_{+}^{n}}\left(\frac{1}{2}R_{ninj,kl}+\frac{1}{3}\text{Sym}_{ij}(\bar{R}_{iksl}R_{nsnj})\right)\frac{t^{2}z_{k}z_{l}z_{i}z_{j}}{\left((1+t)^{2}+|z|^{2}\right)^{n}}dzdt.
\]
Again, by symmetry reasons, we have only to consider the cases $i=j=k=l$,
$i=j\neq k=l$, $i=k\neq j=l$ and $i=l\neq j=k$. Then it is easy
to see that the Symbol term gives no contribution. 

Finally, we have, by (\ref{eq:t2z4})
\begin{multline*}
G_{2}=\int_{\mathbb{R}^{n}}R_{ninj,kl}\frac{t^{2}z_{k}z_{l}z_{i}z_{j}}{\left((1+t)^{2}+|z|^{2}\right)^{n}}dzdt=\sum_{i}R_{nini,ii}\int_{\mathbb{R}^{n}}\frac{t^{2}z_{1}^{4}}{\left((1+t)^{2}+|z|^{2}\right)^{n}}dzdt\\
+\left(\sum_{i\neq k}R_{nini,kk}+\sum_{i\neq j}R_{ninj,ij}+\sum_{i\neq j}R_{ninj,ji}\right)\int_{\mathbb{R}^{n}}\frac{t^{2}z_{1}^{2}z_{2}^{2}}{\left((1+t)^{2}+|z|^{2}\right)^{n}}dzdt\\
=\left(3\sum_{i}R_{nini,ii}+\sum_{i\neq k}R_{nini,kk}+\sum_{i\neq j}R_{ninj,ij}+\sum_{i\neq j}R_{ninj,ji}\right)\frac{1}{n^{2}-1}I_{2}\\
=\left(\sum_{i,k}R_{nini,kk}+\sum_{i,j}R_{ninj,ij}+\sum_{i,j}R_{ninj,ji}\right)\frac{1}{n^{2}-1}I_{2}
\end{multline*}
By \cite[Proof of proposition 3.2, page 1609]{M1} we know $R_{nn,kk}=0$
for all $k=1,\dots,n-1$, so finally we have
\begin{equation}
G_{2}=\frac{(n-2)^{2}}{n^{2}-1}I_{2}R_{ninj,ij}\label{eq:G2}
\end{equation}
Collecting all the terms, by (\ref{eq:G1}), (\ref{eq:G3}), (\ref{eq:G2}),
we have
\begin{align}
L_{1}= & \frac{1}{2}\int_{\mathbb{R}^{n}}|\nabla U|^{2}dzdt+\frac{\delta^{4}}{2}\left(G_{1}+G_{2}+G_{3}\right)\label{eq:L1}\\
= & \frac{1}{2}\int_{\mathbb{R}^{n}}|\nabla U|^{2}dzdt+\nonumber \\
 & +\frac{\delta^{4}}{2}I_{2}\left(\frac{6(n-2)}{n^{2}-1}R_{ninj}^{2}+\frac{(n-2)^{2}}{n^{2}-1}R_{ninj,ij}\right)+O(\delta^{5}).\nonumber 
\end{align}
By (\ref{eq:A2}) and (\ref{eq:L1}), and by Remark \ref{rem:I} we
have 
\begin{align*}
A_{2}+L_{1}= & \frac{1}{2}\int_{\mathbb{R}^{n}}|\nabla U|^{2}dzdt-\delta^{4}\frac{n-2}{96(n-1)^{2}}|\bar{W}(q)|^{2}\int_{\mathbb{R}_{+}^{n}}|z|^{2}U^{2}(z,t)dzdt\\
 & +\frac{\delta^{4}}{2}I_{2}\left(\frac{(n-2)^{2}}{2(n^{2}-1)}\partial_{tt}^{2}R_{\tilde{g}_{q}}+\frac{6(n-2)}{n^{2}-1}R_{ninj}^{2}+\frac{(n-2)^{2}}{n^{2}-1}R_{ninj,ij}\right)\\
= & \frac{1}{2}\int_{\mathbb{R}^{n}}|\nabla U|^{2}dzdt-\delta^{4}\frac{n-2}{96(n-1)^{2}}|\bar{W}(q)|^{2}\int_{\mathbb{R}_{+}^{n}}|z|^{2}U^{2}(z,t)dzdt\\
 & +\frac{\delta^{4}(n-2)}{2(n^{2}-1)}I_{2}\left(\frac{(n-2)}{2}\partial_{tt}^{2}R_{\tilde{g}_{q}}+6R_{ninj}^{2}+(n-2)R_{ninj,ij}\right)+o(\delta^{4})\\
= & \frac{1}{2}\int_{\mathbb{R}^{n}}|\nabla U|^{2}dzdt-\delta^{4}\frac{n-2}{96(n-1)^{2}}|\bar{W}(q)|^{2}\int_{\mathbb{R}_{+}^{n}}|z|^{2}U^{2}(z,t)dzdt\\
 & -\frac{\delta^{4}(n-2)(n-8)}{2(n^{2}-1)}I_{2}R_{ninj}^{2}(q)+o(\delta^{4}).
\end{align*}
In this computation we used the following formula \cite[Formula (3.11) and Proposition 3.2 (5)]{M1}
\begin{align*}
\partial_{tt}^{2}R_{\tilde{g}_{q}} & =-2R_{ninj,ij}-2R_{ninj}^{2}.
\end{align*}
This ends the proof.
\end{proof}

\section{Proof of Theorem \ref{thm:main}: completed}\label{completed}

First of all, we choose $\delta=\lambda\varepsilon^{\frac{1}{3}}$ with $\lambda\in[\alpha,\beta]$
compact subset of $(0,+\infty)$ (so that the second order term in the expansion of \eqref{1claim} have the same rate with respect to $\varepsilon$). Thus, summarizing the result of
Section \ref{sec:red}, we have that, for $\varepsilon$ small, for
any $q\in\partial M$, for any $\lambda\in[\alpha,\beta]$, there
exists a unique $\tilde{\Phi}=\tilde{\Phi}_{\varepsilon,\lambda,q}\in\tilde{K}_{\lambda\varepsilon^{\frac{1}{3}},q}^{\bot}$
which solves (\ref{eq:Pibot}) such that 
\[
\|\tilde{\Phi}\|_{g}=\|\Lambda_{q}\Phi\|_{g}\le C\varepsilon.
\]
Moreover 
\begin{multline*}
J_{\varepsilon,g}(\tilde{W}_{\lambda\varepsilon^{\frac{1}{3}},q}+\lambda^{2}\varepsilon^{\frac{2}{3}}\tilde{V}_{\lambda\varepsilon^{\frac{1}{3}},q}+\tilde{\Phi})=J_{\varepsilon,g}(\tilde{W}_{\lambda\varepsilon^{\frac{1}{3}},q}+\lambda^{2}\varepsilon^{\frac{2}{3}}\tilde{V}_{\lambda\varepsilon^{\frac{1}{3}},q})+O\left(\varepsilon^{5/3}\right)\\
=A+\varepsilon^{\frac{4}{3}}\left[\lambda\gamma(q)B+\lambda^{4}\varphi(q)\right]+O\left(\varepsilon^{5/3}\right)
\end{multline*}
$C^{0}$-uniformly for $q\in\partial M$ and $\lambda\in[\alpha,\beta]$,
where $A,B,\varphi$ are defined in Lemma \ref{lem:espansione}.

Now, setting 
\[
I_{\varepsilon}(\lambda,q):=J_{\varepsilon,g}(\tilde{W}_{\lambda\varepsilon^{\frac{1}{3}},q}+\lambda^{2}\varepsilon^{\frac{2}{3}}\tilde{V}_{\lambda\varepsilon^{\frac{1}{3}},q}+\tilde{\Phi})
\]
 and we can achieve the last part of our Theorem.
\begin{lem}
\label{lem:punticritici}If $(\bar{\lambda},\bar{q})\in(0,+\infty)\times\partial M$
is a critical point for the reduced functional $I_{\varepsilon}(\lambda,q)$,
then the function $\tilde{W}_{\lambda\varepsilon^{\frac{1}{3}},q}+\lambda^{2}\varepsilon^{\frac{2}{3}}\tilde{V}_{\lambda\varepsilon^{\frac{1}{3}},q}+\tilde{\Phi}$
is a solution of (\ref{eq:Ptilde}).
\end{lem}
\begin{proof}
Set $q=q(y)=\psi_{\bar{q}}^{\partial}(y)$. Since $(\bar{\lambda},\bar{q})$
is a critical point for the $I_{\varepsilon}(\lambda,q)$ and since
$\tilde{\Phi}$ is a solution of (\ref{eq:Pibot}) we have, for $h=1,\dots,n-1$,
that there exists $c_{\varepsilon}^{a}\in\mathbb{R}$ such that

\begin{align*}
0= & \left.\frac{\partial}{\partial y_{h}}I_{\varepsilon}(\bar{\lambda},q(y))\right|_{y=0}\\
= & \langle\langle J'_{\varepsilon,g}(\tilde{W}_{\bar{\lambda}\varepsilon^{\frac{1}{3}},q(y)}
+\bar{\lambda}^{2}\varepsilon^{\frac{2}{3}}\tilde{V}_{\bar{\lambda}\varepsilon^{\frac{1}{3}},q(y)}
+\tilde{\Phi})\left.\frac{\partial}{\partial y_{h}}(\tilde{W}_{\bar{\lambda}\varepsilon^{\frac{1}{3}},q(y)}
+\bar{\lambda}^{2}\varepsilon^{\frac{2}{3}}\tilde{V}_{\bar{\lambda}\varepsilon^{\frac{1}{3}},q(y)}
+\tilde{\Phi})\rangle\rangle_{g}\right|_{y=0}\\
= & \sum_{a=1}^{n}c_{\varepsilon}^{a}\left.\langle\langle\Lambda_{q(y)}
Z_{\bar{\lambda}\varepsilon^{\frac{1}{3}},q(y)}^{a},\frac{\partial}{\partial y_{h}}
(\tilde{W}_{\bar{\lambda}\varepsilon^{\frac{1}{3}},q(y)}
+\bar{\lambda}^{2}\varepsilon^{\frac{2}{3}}
\tilde{V}_{\bar{\lambda}\varepsilon^{\frac{1}{3}},q(y)}+\tilde{\Phi})\rangle\rangle_{g}\right|_{y=0}\\
= & \sum_{a=1}^{n}c_{\varepsilon}^{a}\left.\langle\langle\Lambda_{q(y)}Z_{\bar{\lambda}\varepsilon^{\frac{1}{3}},q(y)}^{a},\frac{\partial}{\partial y_{h}}\tilde{W}_{\bar{\lambda}\varepsilon^{\frac{1}{3}},q(y)}\rangle\rangle_{g}\right|_{y=0}\\
 & +\varepsilon^{\frac{2}{3}}\bar{\lambda}^{2}\sum_{a=1}^{n}c_{\varepsilon}^{a}\left.\langle\langle\Lambda_{q(y)}Z_{\bar{\lambda}\varepsilon^{\frac{1}{3}},q(y)}^{a},\frac{\partial}{\partial y_{h}}\tilde{V}_{\bar{\lambda}\varepsilon^{\frac{1}{3}},q(y)}\rangle\rangle_{g}\right|_{y=0}\\
 & \sum_{a=1}^{n}c_{\varepsilon}^{a}\left.\langle\langle\frac{\partial}{\partial y_{h}}\left(\Lambda_{q(y)}Z_{\bar{\lambda}\varepsilon^{\frac{1}{3}},q(y)}^{a}\right),\tilde{\Phi}\rangle\rangle_{g}\right|_{y=0}
\end{align*}
using that
\[
\langle\langle\Lambda_{q(y)}Z_{\bar{\lambda}\varepsilon^{\frac{1}{3}},q(y)}^{a},\frac{\partial}{\partial y_{h}}\tilde{\Phi}\rangle\rangle_{g}=\langle\langle\frac{\partial}{\partial y_{h}}\left(\Lambda_{q(y)}Z_{\bar{\lambda}\varepsilon^{\frac{1}{3}},q(y)}^{a}\right),\tilde{\Phi}\rangle\rangle_{g}
\]
since $\tilde{\Phi}\in K_{\bar{\lambda}\varepsilon^{\frac{1}{3}},q(y)}^{\bot}$
for any $y$.

Arguing as in Lemma 6.1 and Lemma 6.2 of \cite{MP09} we have 
\begin{eqnarray*}
\left\Vert \frac{\partial}{\partial y_{h}}Z_{\bar{\lambda}\varepsilon^{\frac{1}{3}},q(y)}^{a}\right\Vert _{g}=O\left(\frac{1}{\bar{\lambda}\varepsilon^{\frac{1}{3}}}\right) &  & \left\Vert \frac{\partial}{\partial y_{h}}W_{\bar{\lambda}\varepsilon^{\frac{1}{3}},q(y)}\right\Vert _{g}=O\left(\frac{1}{\bar{\lambda}\varepsilon^{\frac{1}{3}}}\right)\\
\left\Vert \frac{\partial}{\partial y_{h}}V_{\bar{\lambda}\varepsilon^{\frac{1}{3}},q(y)}\right\Vert _{g}=O\left(\frac{1}{\bar{\lambda}\varepsilon^{\frac{1}{3}}}\right).
\end{eqnarray*}
For the first term we have 
\begin{align*}
\left.\langle\langle\Lambda_{q(y)}Z_{\bar{\lambda}\varepsilon^{\frac{1}{3}},q(y)}^{a},\frac{\partial}{\partial y_{h}}\tilde{W}_{\bar{\lambda}\varepsilon^{\frac{1}{3}},q}\rangle\rangle_{g}\right|_{y=0}= & \left.\langle\langle\Lambda_{q(y)}Z_{\bar{\lambda}\varepsilon^{\frac{1}{3}},q(y)}^{a},\Lambda_{q(y)}\frac{\partial}{\partial y_{h}}W_{\bar{\lambda}\varepsilon^{\frac{1}{3}},q}\rangle\rangle_{g}\right|_{y=0}\\
 & +\left.\langle\langle\Lambda_{q(y)}Z_{\bar{\lambda}\varepsilon^{\frac{1}{3}},q(y)}^{a},W_{\bar{\lambda}\varepsilon^{\frac{1}{3}},q}\frac{\partial}{\partial y_{h}}\Lambda_{q}\rangle\rangle_{g}\right|_{y=0}
\end{align*}
so we get 
\begin{multline*}
\left.\langle\langle\Lambda_{q(y)}Z_{\bar{\lambda}\varepsilon^{\frac{1}{3}},q(y)}^{a},\Lambda_{q(y)}\frac{\partial}{\partial y_{h}}W_{\bar{\lambda}\varepsilon^{\frac{1}{3}},q}\rangle\rangle_{g}\right|_{y=0}\\
=\frac{1}{\bar{\lambda}\varepsilon^{\frac{1}{3}}}\langle\langle\Lambda_{q}
Z_{\bar{\lambda}\varepsilon^{\frac{1}{3}},q}^{a},Z_{\bar{\lambda}\varepsilon^{\frac{1}{3}},q}^{h})\rangle\rangle_g
+o(1)=\frac{\delta_{ih}}{\bar{\lambda}\varepsilon^{\frac{1}{3}}}+o(1)
\end{multline*}
And, by change of variables, that 
\[
\left.\langle\langle\Lambda_{q(y)}Z_{\bar{\lambda}\varepsilon^{\frac{1}{3}},q(y)}^{a},W_{\bar{\lambda}\varepsilon^{\frac{1}{3}},q}\frac{\partial}{\partial y_{h}}\Lambda_{q}\rangle\rangle_{g}\right|_{y=0}=O(1).
\]
Similarly for the other terms we get 

\begin{align*}
\left.\langle\langle\Lambda_{q(y)}Z_{\bar{\lambda}\varepsilon^{\frac{1}{3}},q(y)}^{a},\frac{\partial}{\partial y_{h}}\tilde{V}_{\bar{\lambda}\varepsilon^{\frac{1}{3}},q(y)}\rangle\rangle_{g}\right|_{y=0} & \le\left\Vert \Lambda_{q(y)}Z_{\bar{\lambda}\varepsilon^{\frac{1}{3}},q(y)}^{a}\right\Vert _{g}\left\Vert \frac{\partial}{\partial y_{h}}\tilde{V}_{\bar{\lambda}\varepsilon^{\frac{1}{3}},q(y)}\right\Vert _{g}=O\left(\frac{1}{\bar{\lambda}\varepsilon^{\frac{1}{3}}}\right)\\
\left.\langle\langle\frac{\partial}{\partial y_{h}}\left(\Lambda_{q(y)}Z_{\bar{\lambda}\varepsilon^{\frac{1}{3}},q(y)}^{a}\right),\tilde{\Phi}\rangle\rangle_{g}\right|_{y=0} & \le\left\Vert \Lambda_{q(y)}Z_{\bar{\lambda}\varepsilon^{\frac{1}{3}},q(y)}^{a}\right\Vert _{g}\left\Vert \tilde{\Phi}\right\Vert _{g}=o(1).
\end{align*}
So we conclude that 
\[
0=\frac{1}{\lambda\varepsilon}\sum_{a=1}^{n}c_{\varepsilon}^{a}\left(\delta_{ih}+O(1)\right)
\]
which implies $c_{\varepsilon}^{a}=0$ for $a=1,\dots,n$.

Analogously we proceed for $\left.\frac{\partial}{\partial\lambda}I_{\varepsilon}(\lambda,\bar{q})\right|_{\lambda=\bar{\lambda}}$,
proving the claim.
\end{proof}
For the sake of completeness, we recall the definition of $C^{0}$-stable
critical point before proving Theorem \ref{thm:main}. 
\begin{defn}
Let $f:\mathbb{R}^{n}\rightarrow\mathbb{R}$ be a $C^{1}$ function
and let $K=\left\{ \xi\in\mathbb{R}^{n}\ :\ \nabla f(\xi)=0\right\} $.
We say that $\xi_{0}\in\mathbb{R}^{n}$ is a $C^{0}$-stable critical
point if $\xi_{0}\in K$ and there exist $\Omega$ neighborhood of
$\xi_{0}$ with $\partial\Omega\cap K=\emptyset$ and a $\eta>0$
such that for any $g:\mathbb{R}^{n}\rightarrow\mathbb{R}$ of class
$C^{1}$ with $\|g-f\|_{C^{0}(\bar{\Omega})}\le\eta$ we have a critical
point of $g$ near $\Omega$.
\end{defn}
We can complete now the proof of Theorem \ref{thm:main}. By Lemma
\ref{lem:punticritici} and by the definition of $C^{0}$-stable critical
point, we have to show that the function 
\[
G(\lambda,q):=\left[\lambda\gamma(q)B+\lambda^{4}\varphi(q)\right]
\]
where $B$ and $\varphi$ are defined in Lemma \ref{lem:espansione},
admits a $C^{0}$-stable critical point. We know that $B>0$ by computation,
and that $\gamma>0$ and $\varphi<0$ by the hypothesis of Th. \ref{thm:main}.
Thus, one can check that there exists $0<\alpha<\beta$ such that
any critical point $(\lambda,q)\in(0,+\infty)\times\partial M$ of
$G$ lies indeed in $(\alpha,\beta)\times\partial M$, because $\frac{\partial G}{\partial\lambda}=B\gamma(q)+4\lambda^{3}\varphi(q)$
and $\frac{\partial G}{\partial\lambda}(\lambda,q)=0$ if and only if
$\lambda^{3}=-\gamma(q)/\varphi(q)>0$. 

Moreover for any number $L<0$ there exists $\bar{\lambda}>0$ such
that $G(\lambda,q)<L$ for any $\lambda>\bar{\lambda}$ and $q\in\partial M$.
Thus there exists a maximum point $(\lambda_{0},q_{0})\in(\alpha,\beta)\times\partial M$
which is $C^{0}$-stable, and we can conclude the proof.
\begin{rem}
\label{rem:esempio}We give another example of function $\gamma(q)$
such that problem (\ref{eq:P}) admits a positive solution. Let $q_{0}\in\partial M$
be a maximum point for $\varphi$. This point exists since $\partial M$
is compact. Now choose $\gamma\in C^{2}(\partial M)$ such that $\gamma$
has a positive local maximum in $q_{0}$. Then the pair $(\lambda_{0},q_{0})=\left(-\sqrt[3]{\frac{B\gamma(q_{0})}{4\varphi(q_{0})}},q_{0}\right)$
is a $C^{0}$-stable critical point for $G(\lambda,q)$. 
\end{rem}
In fact, we have 
\[
\nabla_{\lambda,q}G=(B\gamma(q)+4\lambda^{3}\varphi(q),\lambda B\nabla_{q}\gamma(q)+\lambda^{4}\nabla_{q}\varphi(q))
\]
which vanishes for $(\lambda_{0},q_{0})=\left(-\sqrt[3]{\frac{B\gamma(q_{0})}{4\varphi(q_{0})}},q_{0}\right)$.
Moreover the Hessian matrix is 
\[
G_{\lambda,q}^{''}\left(\lambda_{0},q_{0}\right)=\left(\begin{array}{cc}
2\varphi(q_{0}) & 0\\
0 & \lambda_{0}\gamma''_{q}(q_{0})+\lambda_{0}^{4}\varphi''_{q}(q_{0})
\end{array}\right)
\]
which is negative definite. Thus $(\lambda_{0},q_{0})=\left(-\sqrt[3]{\frac{B\gamma(q_{0})}{4\varphi(q_{0})}},q_{0}\right)$
is a maximum, $C^{0}$-stable, point for $G(\lambda,q)$.

\section{Appendix}

Here we collect the proofs of the technical lemmas we claimed before.
\begin{proof}[Proof of Lemma \ref{lem:vq}.]
 We follow the strategy of \cite[Prop 5.1]{Al}. To prove the existence
of a solution of (\ref{eq:vqdef}) we have to show that the given
term $\left[\frac{1}{3}\bar{R}_{ijkl}(q)z_{k}z_{l}+R_{ninj}(q)t^{2}\right]\partial_{ij}^{2}U$
is $L_{2}(\partial\mathbb{R}_{+}^{n})$-orthogonal to the functions
$j_{1},\dots,j_{n}$. For $l=1,\dots,n-1$ we have 
\begin{multline*}
\int_{\mathbb{R}_{+}^{n}}\left[\frac{1}{3}\bar{R}_{ijkl}(q)z_{k}z_{l}+R_{ninj}(q)t^{2}\right]\partial_{ij}^{2}Uj_{b}\\
=\int_{\mathbb{R}_{+}^{n}}\left[\frac{1}{3}\bar{R}_{ijkl}(q)z_{k}z_{l}+R_{ninj}(q)t^{2}\right]\partial_{ij}^{2}U\partial_{l}Udzdt=0
\end{multline*}
by symmetry, since the integrand is odd with respect to the $z$ variables. 

For the last term, since when $i\neq j$ we have 
\[
\partial_{ij}U=\frac{n(n-2)z_{i}z_{j}}{\left((1+t)^{2}+|z|^{2}\right)^{\frac{n+2}{2}}}
\]
and since when $i=j$ we have $\bar{R}_{iikl}=0$ and, by (\ref{eq:Ricci}),
$R_{nini}=R_{nn}=0$ we have 
\begin{multline*}
\int_{\mathbb{R}_{+}^{n}}\left[\frac{1}{3}\bar{R}_{ijkl}(q)z_{k}z_{l}+R_{ninj}(q)t^{2}\right]\partial_{ij}^{2}UUdzdt\\
=\sum_{i\neq j}\sum_{k}\int_{\mathbb{R}_{+}^{n}}\left[\frac{1}{3}\bar{R}_{ijkl}(q)z_{k}z_{l}+R_{ninj}(q)t^{2}\right]\frac{n(n-2)z_{i}z_{j}}{\left((1+t)^{2}+|z|^{2}\right)^{n}}
\end{multline*}
and since $i\neq j$, by symmetry all the terms containing $t^{2}z_{i}z_{j}$
vanish and the others terms are non zero only when $i=k$ and $j=l$
or when $j=k$ and $i=l$, thus 
\[
\int_{\mathbb{R}_{+}^{n}}\left[\frac{1}{3}\bar{R}_{ijkl}(q)z_{k}z_{l}+R_{ninj}(q)t^{2}\right]\partial_{ij}^{2}UUdzdt
\]
\[
=\sum_{k}\int_{\mathbb{R}_{+}^{n}}\left[\frac{1}{3}\bar{R}_{klkl}(q)+\frac{1}{3}\bar{R}_{lkkl}(q)\right]\frac{n(n-2)z_{k}^{2}z_{l}^{2}}{\left((1+t)^{2}+|z|^{2}\right)^{-n}}=0
\]
since $\bar{R}_{klkl}(q)=-\bar{R}_{lkkl}(q)$. Moreover
\begin{multline*}
\int_{\mathbb{R}_{+}^{n}}\left[\frac{1}{3}\bar{R}_{ijkl}(q)z_{k}z_{l}+R_{ninj}(q)t^{2}\right]\partial_{ij}^{2}Uy_{b}\partial_{b}Udtz\\
=n(2-n)\sum_{i\ne j}\sum_{k,s}\int_{\mathbb{R}_{+}^{n}}\left[\frac{1}{3}\bar{R}_{ijkl}(q)z_{k}z_{l}+R_{ninj}(q)t^{2}\right]\frac{z_{i}z_{j}\left(z_{s}z_{s}+t(1+t)\right)}{\left((1+t)^{2}+|z|^{2}\right)^{-n-1}}\\
=n(2-n)\sum_{k}\int_{\mathbb{R}_{+}^{n}}\left[\frac{1}{3}\bar{R}_{klkl}(q)+\frac{1}{3}\bar{R}_{lkkl}(q)\right]\frac{z_{k}^{2}z_{l}^{2}\left(\sum_{s}z_{s}z_{s}+t(1+t)\right)}{\left((1+t)^{2}+|z|^{2}\right)^{-n-1}}=0.
\end{multline*}
Then there exists a solution. Also there exists a unique solution
$v_{q}$ which is $L_{2}(\partial\mathbb{R}_{+}^{n})$-orthogonal
to $j_{b}$ for $b=1,\cdots,n$. 

To prove the estimates (\ref{eq:Uvq}) and (\ref{new}) we use the
inversion $F:\mathbb{R}_{+}^{n}\rightarrow B^{n}\smallsetminus\left\{ (0,\dots,0-1)\right\} $
where $B^{n}\subset\mathbb{R}^{n}$ is the closed ball centered in
$(0,\dots,0,-1/2)$ and radius $1/2$. The explicit expression for
$F$ is 
\[
F(y_{1},\dots,y_{n})=\frac{(y_{1},\dots,y_{n-1},y_{n}+1)}{y_{1}^{2}+\dots+y_{n-1}^{2}+(y_{n}+1)^{2}}+(0,\dots,0-1).
\]
We set
\[
f_{q}(F(y))=\left[\frac{1}{3}\bar{R}_{ijkl}(q)y_{k}y_{l}+R_{ninj}(q)y_{n}^{2}\right]\partial_{ij}^{2}U(y)U^{-\frac{n+2}{n-2}}(y).
\]
By direct computation we have $|f_{i}(F(y))|\le C(1+|y|)^{4}$, so
we have 
\begin{equation}
|f_{q}(\xi)|\le C\left(1+\frac{1}{|\xi|}\right)^{4}\le C\frac{1}{\left(1+|\xi|\right)^{4}}\label{eq:stimafi}
\end{equation}
So it is possible to smoothly extend $f_{q}$ to the whole $B^{n}$,
and it turns out that if $v_{q}$ solves (\ref{eq:vqdef}), then \textbf{$\bar{v}_{q}:=(U^{-1}v_{q})\circ F^{-1}$
}solves 
\begin{equation}
\left\{ \begin{array}{ccc}
-\Delta\bar{v}=f_{q} &  & \text{on }B^{n}\\
\frac{\partial v}{\partial y_{n}}+2\bar{v}=0 &  & \text{on }\partial B^{n}
\end{array}\right..\label{eq:vbardef}
\end{equation}
Then existence and uniqueness of $\bar{v}_{q}$ are standard. To prove
the decadence estimates, fixed $w\in B^{n}$, consider the Green's
function $G(\xi,w)$ with boundary condition $\left(\frac{\partial}{\partial\nu}+2\right)G=0$.
Then by Green's formula and by (\ref{eq:vbardef}) we have 
\[
\bar{v}_{q}(\xi)=\int_{B^{n}}G(\xi,w)\Delta\bar{v}_{q}(\xi)+\int_{\partial B^{n}}\bar{v}_{q}\frac{\partial}{\partial\nu}G-G\frac{\partial}{\partial\nu}\bar{v}_{q}=-\int_{B^{n}}G(\xi,w)f_{q}(\xi)
\]
 and, in light of (\ref{eq:stimafi}) we have 
\[
|\bar{v}_{q}(\xi)|\le C\int_{B^{n}}|\xi-w|^{2-n}\left(1+|\xi|\right)^{-4}
\]
and by \cite[Prop 4.12 page 108]{Au} that $|\bar{v}_{q}(\xi)|\le C\left(1+|\xi|\right)^{-2}$
and by the definition of $\bar{v}_{q}$ we deduce 
\[
|v_{q}(y)|\le C\left(1+|y|\right)^{4-n}.
\]
The estimates on the first and the second derivatives of $v_{q}$
can be achieved in a similar way. 

It remains to prove (\ref{eq:Uvq}) and (\ref{new}). Notice that,
changing of variables and proceeding as at the beginning of this proof,
we have 
\[
\int_{B^{n}}f_{q}(\xi)d\xi=\int_{\mathbb{R}_{+}^{n}}\frac{\left[\frac{1}{3}\bar{R}_{ijkl}(q)y_{k}y_{l}+R_{ninj}(q)y_{n}^{2}\right]\partial_{ij}^{2}U(y)U^{-\frac{n+2}{n-2}}(y)}{(y_{1}^{2}+\dots+y_{n-1}^{2}+(y_{n}+1)^{2})^{n}}dy=0.
\]
So we have, using (\ref{eq:vbardef}) and integrating by parts, that
\begin{equation}
0=\int_{B^{n}}f_{q}=-\int_{B^{n}}\Delta\bar{v}_{q}=-\int_{\partial B^{n}}\frac{\partial}{\partial\nu}\bar{v}_{q}=-\int_{\partial B^{n}}2\bar{v}_{q}\label{eq:fq0}
\end{equation}
and, changing variables again, 
\begin{multline*}
0=\int_{\partial B^{n}}2\bar{v}_{q}(\xi)d\xi_{1}\dots d\xi_{n-1}=\int_{\partial\mathbb{R}_{+}^{n}}U^{-1}(y)v_{q}(y)U^{\frac{2(n-1)}{n-2}}(y)dy_{1}\dots dy_{n-1}\\
=\int_{\partial\mathbb{R}_{+}^{n}}U^{\frac{n}{n-2}}(y)v_{q}(y)dy_{1}\dots dy_{n-1}.
\end{multline*}
It is known (see \cite{Al}), that it holds, on $H^{1}(B^{n})$, 
\[
\inf_{\int_{\partial B_{n}}\phi=0}\frac{\int_{B^{n}}|\nabla\phi|^{2}}{\int_{\partial B^{n}}|\phi|^{2}}=2.
\]
Since, by (\ref{eq:fq0}), we know that $\int_{\partial B^{n}}\bar{v}_{q}=0$,
we get
\[
2\int_{\partial B^{n}}\bar{v}_{q}^{2}\le\int_{B^{n}}|\nabla\bar{v}_{q}|^{2},
\]
so, integrating by parts
\[
-\int_{B^{n}}\bar{v}_{q}\Delta\bar{v}_{q}=\int_{B^{n}}|\nabla\bar{v}_{q}|^{2}-2\int_{\partial B^{n}}\bar{v}_{q}^{2}\ge0.
\]
 By the properties of the inversion $F$ (see \cite[formula (5.10)]{Al})
we have also 
\[
-\int_{B^{n}}\bar{v}_{q}\Delta\bar{v}_{q}=-\int_{\mathbb{R}_{+}^{n}}v_{q}\Delta v_{q}.
\]
Finally, we want to prove that $v_{q}\in C^{2}(\partial M)$. Let
$q_{0}\in\partial M$. If $q\in\partial M$ is sufficiently close
to $q_{0}$, in Fermi coordinates we have $q=q(\eta)=\exp_{q_{0}}\eta$,
with $\eta\in\mathbb{R}^{n-1}$. So $v_{q}=v_{\exp_{q_{0}}\eta}$
and we define 
\[
\Gamma_{i}=\left.\frac{\partial}{\partial y_{i}}v_{\exp_{q_{0}}\eta}\right|_{\eta=0}.
\]
We prove the result for $\Gamma_{1}$, being the other cases completely
analogous. By (\ref{eq:vqdef}) we have that $\Gamma_{1}$ solves
\[
\left\{ \begin{array}{ccc}
-\Delta\Gamma_{1}=\left[\frac{1}{3}\left.\frac{\partial}{\partial\eta_{1}}\left(\bar{R}_{ijkl}(q(y))\right)\right|_{y=0}y_{k}y_{l}+\left.\frac{\partial}{\partial\eta_{1}}\left(\bar{R}_{ninj}(q(y))\right)\right|_{y=0}\right]\partial_{ij}^{2}U &  & \text{on }\mathbb{R}_{+}^{n};\\
\frac{\partial\Gamma_{1}}{\partial t}+nU^{\frac{2}{n-2}}\Gamma_{1}=0 &  & \text{on \ensuremath{\partial}}\mathbb{R}_{+}^{n}.
\end{array}\right.
\]
and, since $\frac{\partial R_{nn}}{\partial\eta_{i}}(q)=0$ (see \cite[Prop 3.2 (4)]{M1}),
we can proceed as at the beginning of this proof to show that $\Gamma_{1}$
exists. Analogously we get the claim for the second derivative. 

That concludes the proof.
\end{proof}
\begin{proof}[Proof of Lemma \ref{lem:JWpiuPhi}.]
 By (\ref{eq:Jlambda}) we estimate, for some $\theta\in(0,1)$ 
\begin{multline*}
\tilde{J}_{\varepsilon,\tilde{g}_{q}}(W_{\delta,q}+\delta^{2}V_{\delta,q}+\Phi)-\tilde{J}_{\varepsilon,\tilde{g}_{q}}(W_{\delta,q}+\delta^{2}V_{\delta,q})=\tilde{J}_{\varepsilon,\tilde{g}_{q}}'(W_{\delta,q}+\delta^{2}V_{\delta,q})[\Phi]\\
+\frac{1}{2}\tilde{J}_{\varepsilon,\tilde{g}_{q}}''(W_{\delta,q}+\delta^{2}V_{\delta,q}+\theta\Phi)[\Phi,\Phi]\\
=\int_{M}\left(\nabla_{\tilde{g}_{q}}W_{\delta,q}+\delta^{2}\nabla_{\tilde{g}_{q}}V_{\delta,q}\right)\nabla_{\tilde{g}_{q}}\Phi+\frac{n-2}{2(n-1)}R_{\tilde{g}_{q}}\left(W_{\delta,q}+\delta^{2}V_{\delta,q}\right)\Phi d\mu_{\tilde{g}_{q}}\\
+\int_{\partial M}\varepsilon\gamma\left(W_{\delta,q}+\delta^{2}V_{\delta,q}\right)\Phi d\sigma_{\tilde{g}_{q}}-(n-2)\int_{\partial M}\left(\left(W_{\delta,q}+\delta V_{\delta,q}\right)^{+}\right)^{\frac{n}{n-2}}\Phi d\sigma_{\tilde{g}_{q}}\\
+\frac{n-2}{2}\int_{\partial M}h_{\tilde{g}_{q}}\left(W_{\delta,q}+\delta^{2}V_{\delta,q}\right)\Phi d\sigma_{\tilde{g}_{q}}\\
+\frac{1}{2}\int_{M}|\nabla_{\tilde{g}_{q}}\Phi|^{2}+\frac{n-2}{4(n-1)}R_{\tilde{g}_{q}}\Phi^{2}d\mu_{\tilde{g}_{q}}+\frac{1}{2}\int_{\partial M}\varepsilon\gamma\Phi^{2}d\sigma_{\tilde{g}_{q}}\\
+\frac{n-2}{4}\int_{\partial M}h_{\tilde{g}_{q}}\Phi^{2}d\sigma_{\tilde{g}_{q}}-\frac{n}{2}\int_{\partial M}\left(\left(W_{\delta,q}+\delta V_{\delta,q}+\theta\Phi\right)^{+}\right)^{\frac{2}{n-2}}\Phi^{2}d\sigma_{\tilde{g}_{q}}.
\end{multline*}
Immediately we have, by Holder inequality,
\[
\int_{M}|\nabla_{\tilde{g}_{q}}\Phi|^{2}+\frac{n-2}{4(n-1)}R_{\tilde{g}_{q}}\Phi^{2}d\mu_{\tilde{g}_{q}}+\int_{\partial M}\left(\varepsilon\gamma+\frac{n-2}{4}h_{\tilde{g}_{q}}\right)\Phi^{2}d\sigma\le C\|\Phi\|_{\tilde{g}_{q}}^{2}=C\|\tilde{\Phi}\|_{g}^{2};
\]
\[
\int_{M}\frac{n-2}{2(n-1)}R_{\tilde{g}_{q}}W_{\delta,q}\Phi d\mu_{\tilde{g}_{q}}\le C\|W_{\delta,q}\|_{L^{\frac{2n}{n+2}}(M,\tilde{g}_{q})}\|\Phi\|_{L^{\frac{2n}{n-2}}(M,\tilde{g}_{q})}\le C\delta^{2}\|\tilde{\Phi}\|_{g};
\]
\[
\delta^{2}\int_{M}aV_{\delta,q}\Phi d\mu_{\tilde{g}_{q}}\le C\delta^{2}\|V_{\delta,q}\|_{L^{2}(M,\tilde{g}_{q})}\|\Phi\|_{L^{2}(M,\tilde{g}_{q})}\le C\delta^{2}\|\tilde{\Phi}\|_{g};
\]
\begin{align*}
\int_{\partial M}\varepsilon\tilde{\gamma}\left(W_{\delta,q}+\delta^{2}V_{\delta,q}\right)\Phi d\sigma_{\tilde{g}_{q}} & \le C\varepsilon\|W_{\delta,q}+\delta V_{\delta,q}\|_{L^{\frac{2(n-1)}{n}}(\partial M,\tilde{g}_{q})}\|\Phi\|_{L^{\frac{2(n-1)}{n-2}}(\partial M,\tilde{g}_{q})}\\
 & \le C\varepsilon\delta\|\tilde{\Phi}\|_{g}
\end{align*}
\begin{align*}
\int_{\partial M}\left(\left(W_{\delta,q}+\delta V_{\delta,q}+\theta\Phi\right)^{+}\right)^{\frac{2}{n-2}}\Phi^{2}d\sigma_{\tilde{g}_{q}} & \le C\left(\left\Vert W_{\delta,q}+\delta V_{\delta,q}+\theta\Phi\right\Vert _{L^{\frac{2(n-1)}{n-2}}(\partial M,\tilde{g}_{q})}^{\frac{2}{n-2}}\right)\|\Phi\|_{\tilde{g}_{q}}^{2}\\
 & \le C\|\tilde{\Phi}\|_{g}^{2};
\end{align*}
By integration by parts we have 
\begin{multline*}
\int_{M}\left(\nabla_{\tilde{g}_{q}}W_{\delta,q}+\delta^{2}\nabla_{\tilde{g}_{q}}V_{\delta,q}\right)\nabla_{\tilde{g}_{q}}\Phi d\mu_{\tilde{g}_{q}}=-\int_{M}\Delta_{\tilde{g}_{q}}\left(W_{\delta,q}+\delta^{2}V_{\delta,q}\right)\Phi d\mu_{\tilde{g}_{q}}\\
+\int_{\partial M}\left(\frac{\partial}{\partial\nu}W_{\delta,q}+\delta^{2}\frac{\partial}{\partial\nu}V_{\delta,q}\right)\Phi d\sigma_{\tilde{g}_{q}}.
\end{multline*}
and, as in (\ref{eq:deltaw+v}) we get 
\[
\int_{M}\Delta_{\tilde{g}_{q}}\left(W_{\delta,q}+\delta^{2}V_{\delta,q}\right)\Phi d\mu_{\tilde{g}_{q}}\le\|\Delta_{\tilde{g}_{q}}(W_{\delta,q}+\delta^{2}V_{\delta,q})\|_{L^{\frac{2n}{n+2}}(M,\tilde{g}_{q})}\|\Phi\|_{\tilde{g}_{q}}=O(\delta^{2})\|\tilde{\Phi}\|_{g}.
\]
Moreover, by Holder inequality, 
\[
\int_{\partial M}\delta^{2}\frac{\partial}{\partial\nu}V_{\delta,q}\Phi d\mu_{\tilde{g}_{q}}\le\delta^{2}\left\Vert \frac{\partial}{\partial\nu}V_{\delta,q}\right\Vert _{L^{\frac{2(n-1)}{n}}(\partial M,\tilde{g}_{q})}\|\Phi\|_{L^{\frac{2(n-1)}{n-2}}(\partial M,\tilde{g}_{q})}=O(\delta^{2})\|\tilde{\Phi}\|_{g}.
\]
Since $\partial M$ is umbilic, proceeding as in (\ref{eq:hresto1})
(\ref{eq:hresto2}), we get 
\[
\int_{\partial M}h_{\tilde{g}_{q}}(W_{\delta,q}+\delta^{2}V_{\delta,q})\Phi d\sigma_{\tilde{g}_{q}}=O(\delta^{4})\|\tilde{\Phi}\|_{g}.
\]
In the end we need to verify that 
\begin{multline*}
\int_{\partial M}\left[(n-2)\left(\left(W_{\delta,q}+\delta^{2}V_{\delta,q}\right)^{+}\right)^{\frac{n}{n-2}}-\frac{\partial}{\partial\nu}W_{\delta,q}\right]\Phi d\sigma_{\tilde{g}_{q}}\\
=\left\Vert (n-2)\left(\left(W_{\delta,q}+\delta^{2}V_{\delta,q}\right)^{+}\right)^{\frac{n}{n-2}}-\frac{\partial}{\partial\nu}W_{\delta,q}\right\Vert _{L^{\frac{2(n-1)}{n}}(\partial M,\tilde{g}_{q})}\|\Phi\|_{L^{\frac{2(n-1)}{n-2}}(\partial M,\tilde{g}_{q})}\\
=o(\delta^{2})\|\tilde{\Phi}\|_{g}
\end{multline*}
In fact, by (\ref{eq:Udelta}) and by taylor expansion we have 
\begin{multline*}
\int_{\partial M}\left[(n-2)\left(\left(W_{\delta,q}+\delta^{2}V_{\delta,q}\right)^{+}\right)^{\frac{n}{n-2}}-\frac{\partial}{\partial\nu}W_{\delta,q}\right]^{\frac{2(n-1)}{n}}d\sigma_{\tilde{g}_{q}}\\
\le\int_{\partial\mathbb{R}_{+}^{n}}\left[(n-2)\left(\left(U_{\delta}+\delta^{2}\left(v_{q}\right)_{\delta}\right)^{+}\right)^{\frac{n}{n-2}}+\frac{\partial}{\partial t}U_{\delta}\right]^{\frac{2(n-1)}{n}}dz+o(\delta^{\frac{4(n-1)}{n}})\\
\le\int_{\partial\mathbb{R}_{+}^{n}}\left[n\left(\left(U_{\delta}+\theta\delta^{2}\left(v_{q}\right)_{\delta}\right)^{+}\right)^{\frac{2}{n-2}}\delta^{2}\left(v_{q}\right)_{\delta}\right]^{\frac{2(n-1)}{n}}dz+o(\delta^{\frac{4(n-1)}{n}})=o(\delta^{\frac{4(n-1)}{n}}),
\end{multline*}
which concludes the proof. 
\end{proof}

\end{document}